\documentclass[11pt,a4paper,reqno]{amsart}
\usepackage[T1]{fontenc}
\usepackage{amsmath}
\usepackage{amsthm}  
\usepackage{amsfonts}   
\usepackage{amssymb} 
\usepackage{amsbsy}
\usepackage{mathrsfs} 
\usepackage{bbm}  
\usepackage{titletoc}      
\newtheorem{theorem}{Theorem}[section] 

\newtheorem{lemma}[theorem]{Lemma} 

\newtheorem{proposition}[theorem]{Proposition}

\theoremstyle{remark}
\newtheorem{remark}[theorem]{\it \bf{Remark}\/}
 
\numberwithin{equation}{section}
\catcode`@=11 
\def\section{\@startsection{section}{1}%
  \z@{1.5\linespacing\@plus\linespacing}{.5\linespacing}%
  {\normalfont\bfseries\large\centering}}
\catcode`@=12
\newcommand{\be}{\begin{equation}}
\newcommand{\ee}{\end{equation}}
\newcommand{\bea}{\begin{eqnarray}} 
  \newcommand{\eea}{\end{eqnarray}} 
\newcommand{\bee}{\begin{eqnarray*}}
\newcommand{\eee}{\end{eqnarray*}}

\def\NN{\mathbb{N}}
\def\RR{\mathbb{R}}

\def\ds{\displaystyle}
\def\ni{\noindent} 
\def\bs{\bigskip}
\def\ms{\medskip}
\def\ss{\smallskip}
\def\eps{\varepsilon}
\def\fref#1{{\rm (\ref{#1})}}
\def\pref#1{{\rm \ref{#1}}}

\def\calL{{\mathcal L}}
\def\calC{{\mathcal C}}

\def\calA{{\mathcal A}}
\def\calE{{\mathcal E}}
\def\calO{{\mathcal O}}

\def\calJ{{\mathcal J}}
\def\calK{{\mathcal K}}

\def\Hunper{{\rm H}^1_{per}}

\def\H{{\rm H}}  
\catcode`@=11
\def\supess{\mathop{\operator@font Sup\,ess}}
\catcode`@=12
\DeclareMathOperator{\Tr}{Tr}

\def\un{{\mathbbmss{1}}} 
\title[]{An inverse problem in\\ quantum statistical physics}
\author[F. M\'ehats]{Florian M\'ehats}
\address[F. M\'ehats]{IRMAR, Universit\'e de Rennes 1 and IPSO, INRIA Rennes, France}
\email{florian.mehats@univ-rennes1.fr}
\author[O. Pinaud]{Olivier Pinaud}
\address[O. Pinaud]{Institut Camille Jordan, ISTIL, Universit\'e de Lyon 1, France}
\email{pinaud@math.univ-lyon1.fr}
\begin{document}

\begin{abstract}
We address the following inverse problem in quantum statistical physics: does the quantum free energy (von Neumann entropy + kinetic energy) admit a unique minimizer among the density operators having a given local density $n(x)$? We give a positive answer to that question, in dimension one. This enables to define rigourously the notion of local quantum equilibrium, or quantum Maxwellian, which is at the basis of recently derived quantum hydrodynamic models and quantum drift-diffusion models. We also characterize this unique minimizer, which takes the form of a global thermodynamic equilibrium (canonical ensemble) with a quantum chemical potential. 
\end{abstract}


\maketitle
\titlecontents{section} 
[1.5em] 
{\vspace*{0.1em}\bf} 
{\contentslabel{2.3em}} 
{\hspace*{-2.3em}} 
{\titlerule*[0.5pc]{\rm.}\contentspage\vspace*{0.1em}}

\vspace*{-2mm}


\section{Introduction}

We deal with a question which is at the core of recently derived quantum hydrodynamic models based on an entropy minimization principle \cite{DR,QET}. Let a given density of particles $n(x)\geq 0$, can we find a minimizer of the quantum free energy among the density operators $\varrho$ having $n(x)$ as local density, i.e. satisfying the constraint $\rho(x,x)=n(x)$, where $\rho(x,y)$ denotes the integral kernel of $\varrho$?

This question arises in the moment closure strategy initially introduced by Degond and Ringhofer in \cite{DR} in order to derive quantum hydrodynamic models from first principles. Let us briefly review this theory (for more details, one can refer to the reviews \cite{QHD-review,DGMRlivre}). The quest of macroscopic quantum models is motivated by applications such as nanoelectronics, where affordable numerical simulations of the electronic transport are necessary while the miniaturization of devices now imposes to take into account quantum mechanical effects in the models, resulting in a higher simulation cost. At the microscopic level of description, the Schr\"odinger equation and the quantum Liouville equation are numerically too expensive, which motivates the derivation of models at a more macroscopic level. In the classical setting, the relationships between microscopic (kinetic) and macroscopic (fluid) levels of description are fairly well understood by means of asymptoti
 c analysis, see for instance \cite{golse-lsr,golse,ovy}. In particular, it is known that the understanding of the structure of the fluid model relies on the properties of the collision operator at the underlying kinetic level. Indeed, collisions are the source of entropy dissipation, which induces the relaxation of the system towards local thermodynamical equilibria. The free parameters of these local equilibria are the moments of the system (e.g. local density, momentum and energy) and are driven by the fluid equations. Arguing that the derivation of precise quantum collision operators is a very difficult task, while only the macroscopic properties of such operators is needed in our context, Degond and Ringhofer have grounded their theory on a notion of quantum local equilibria. To do so, they have generalized Levermore's moment approach \cite{levermore} to the quantum setting. The idea consists in closing the system of moment equations by defining a local equilibrium as th
 e minimizer of an entropy functional (say, the von Neumann entropy) under moment constraints. 

In \cite{QET}, this approach was adapted so as to describe systems in strong interaction with their surrounding media and obtain quantum macroscopic models by applying a diffusive asymptotics. The most simple of these models, the quantum drift-diffusion model, was studied numerically in \cite{QDD-SIAM,QDD-JCP} and the simulation results for one-dimensional devices such as resonant tunneling diodes were encouraging.  This model is based on the most elementary constrained entropy minimization problem. Indeed, in this case, the local quantum equilibrium at a given temperature, also called quantum Maxwellian, is defined as the minimizer of the quantum free energy subject to a local constraint of prescribed density. Note that not only the total number of particles is fixed, as in the usual quantum statistics theory (for the so-called canonical ensemble), but also the local density $n(x)$ is imposed at any point $x$ of the physical space. This problem has been studied formally in \
 cite{QET} and the Lagrange multipliers theory lead to the existence of a quantum chemical potential $A(x)$ such that the solution of the minimization problem is a density operator of the form
\be
\label{ce}
\varrho=\exp\left(-\frac{-\Delta+A(x)}{T}\right).
\ee
Remark that the difficulty in this problem lies in the fact that its solution will depend on its data in a global way. The similar problem in classical physics, i.e. reconstructing $f(x,v)=\exp(-(\frac{|v|^2}{2}+A(x)))$ from its density $n(x)=\int f(x,v)dv$, is very simple and the chemical potential, given by $A(x)=-\log n(x)+\frac{3}{2}\log (2\pi)$, depends on $n(x)$ in a local way. Here, due to the operator formalism of quantum mechanics, which is not commutative, the density and the associated chemical potential are linked together by a non-explicit formula, and in a global manner. 

To end this short presentation, let us also recall that this quantum drift-diffusion model displays formally several interesting properties: it dissipates a quantum fluid entropy, which indicates that it should be mathematically well-posed, and it can be related to other known models after some approximations (for instance, semiclassical expansions on the quantum drift-diffusion system enable to derive the density-gradient model). Besides, a whole family of quantum fluid models were derived by several authors, based on the same entropy minimization principle: quantum Spherical Harmonic Expansion (QSHE) models \cite{QSHE}, quantum isothermal Euler systems \cite{jungel-matthes,isotherme}, quantum hydrodynamics \cite{jungel-matthes-milisic,QHD-CMS}, models with viscosity \cite{brull-mehats,jungel,jungel2}, quantum models for systems such as subbands \cite{ringhofer} or spins \cite{barletti-mehats}. Nevertheless, one has to put the emphasis on the fact that all these studies rema
 in yet {\em at a formal level}. Even the notion of local quantum equilibrium has only been defined formally and this problem of entropy minimization under local constraints is widely open.

The aim of this paper is to make a first step towards the rigorous justification of these models, by studying the quantum entropy minimization principle in the most simple situation, in the case of a density constraint. We work in dimension one, in a finite box with periodic boundary conditions. Our main result, Theorem \ref{theo1}, is presented after a few notations in the next section. We show that, in an appropriate functional framework, the quantum Maxwellian is properly defined, i.e. that to any density $n(x)>0$ corresponds a unique density matrix $\varrho$ minimizing the free energy. Moreover, we prove that $\varrho$ actually takes the form \fref{ce}, where $A(x)$ is the quantum chemical potential (in the sequel of the paper, the temperature $T$ will be set to 1).

Let us now make a remark. One can see on the formula \fref{ce} that the quantum Maxwellian reads as the global equilibrium canonical ensemble associated to the Hamiltonian $-\Delta+A(x)$, where the chemical potential $A(x)$ is seen as an applied potential. Hence, our problem can be reformulated as the following inverse problem in quantum statistical mechanics. Let a system at thermal equilibrium with a surrounding media at a given temperature, in a certain potential. Can we reconstruct the potential from the measurement of the density at any point? This problem has been much less studied than more standard inverse problems such as the inverse scattering theory (reconstructing the potential from its scattering effects) or the inverse spectral problem (reconstructing the potential from the spectrum of the associated Hamiltonian). Nevertheless, one can quote at least two references where similar inverse problems have been investigated. In \cite{lemm} (see also the series of ref.
  5 therein), a practical method for reconstructing potentials from measurements was settled using Feynman path integrals and, in \cite{hugenholtz}, a close problem for quantum spin systems was studied.

The outline of this paper is as follows. In Section \ref{sect2}, we define the functional framework of the paper and state our main theorem. In Section \ref{sect3}, we study the entropy and the free energy and give some useful results for the sequel. In Section \ref{sect4}, we prove the existence and uniqueness of the minimizer $\varrho[n]$ associated to a density $n$. Section \ref{sect5} is devoted to the characterization of $\varrho[n]$ via the Euler-Lagrange equation for the minimization problem. To deal with the constraint, we introduce a penalized problem.

Future developments of this work will involve several directions. A first extension will concern the investigation of other spatial configurations: other boundary conditions, whole-space case, or space dimension greater than one. We will also investigate the entropy minimization problem with constraints of higher order moments. As in the case of classical physics, it might lead to ill-posed problems and to delicate problems of moment realizability \cite{junk}. Another interesting question which remains to be solved concerns the quantum evolution: can we define an evolution for a quantum Liouville equation with a BGK-relaxation operator based on the local equilibria defined in this paper, as for instance in \cite{arnold} for other relaxation operators? This issue is linked to the possibility of  rigourously deriving the quantum drift-diffusion model.

\section{Notations and main result}
\label{sect2}
Let us describe the functional framework of this paper. The physical space that we consider is monodimensional and bounded. The particles are supposed to be confined in the torus $[0,1]$, i.e. with periodic boundary conditions. We consider the Hamiltonian $$H=-\frac{d^2}{dx^2}$$ on the space $L^2(0,1)$ of complex-valued functions, equipped with the domain
$$D(H)=\left\{u\in \H^2(0,1):\,u(0)=u(1),\,\frac{du}{dx}(0)=\frac{du}{dx}(1)\right\}.$$
The domain of the associated quadratic form is
$$\Hunper=\left\{u\in \H^1(0,1): \,u(0)=u(1)\right\}.$$
Its dual space will be denoted $\H^{-1}_{per}$. Remark that one has the following identification:
\be
\forall u,v \in \Hunper,\qquad (\sqrt{H}u,\sqrt{H}v)=\left(\frac{du}{dx},\frac{dv}{dx}\right),\quad \|\sqrt{H}u\|_{L^2}=\left\|\frac{du}{dx}\right\|_{L^2}.
\label{ident}
\ee
We shall denote by $\calJ_1$ the space of trace class operators on $L^2(0,1)$ \cite{RS-80-I,Simon-trace} and by $\calJ_2$ the space of Hilbert-Schmidt operators on $L^2(0,1)$, which are both ideals of the space $\calL(L^2(0,1))$ of bounded operators on $L^2(0,1)$. We denote by $\calK$ the space of compact operators on $L^2(0,1)$.

A density operator is defined as a nonnegative trace class self-adjoint operator on $L^2(0,1)$. Let us define the following space:
$$\calE=\left\{\varrho\in \calJ_1,\,\varrho=\varrho^*\mbox{ and }\sqrt{H}|\varrho|\sqrt{H}\in \calJ_1\right\}.$$
This is a Banach space endowed with the norm
$$\|\varrho\|_{\calE}=\Tr |\varrho|+\Tr(\sqrt{H}|\varrho|\sqrt{H}).$$
For any $\varrho\in \calE$, the associated density $n[\varrho]$ is formally defined by
$$n[\varrho](x)=\rho(x,x),$$
where $\rho$ is the integral kernel of $\varrho$ satisfying
$$\forall \phi\in L^2(0,1),\quad \varrho(\phi)(x)=\int_0^1\rho(x,y)\phi(y)dy.$$
The density $n[\varrho]$ can be in fact identified by the following weak formulation:
\be
\label{weak}
\forall \Phi\in L^\infty(0,1),\quad \Tr (\Phi \varrho)=\int_0^1\Phi(x)n[\varrho](x)dx,
\ee
where, in the left-hand side, $\Phi$ denotes the multiplication operator by $\Phi$, which belongs to $\calL(L^2(0,1))$.
If the spectral decomposition of $\varrho$ is written
$$\varrho=\sum_{k=1}^\infty\rho_k\,(\phi_k,\cdot)_{L^2}\,\phi_k$$
then we have
\be
\label{srt1}
\Tr \sqrt{H}|\varrho|\sqrt{H}=\|\sqrt{H}\sqrt{|\varrho|}\|^2_{\calJ_2}=\sum_{k=1}^\infty|\rho_k|\left\|\frac {d\phi_k}{dx}\right\|_{L^2}^2
\ee
\be
\label{srt2}
n[\varrho](x)=\sum_{k=1}^\infty\rho_k|\phi_k(x)|^2,\qquad \|n[\varrho]\|_{L^1}\leq \sum_{k=1}^\infty |\rho_k|=\Tr |\varrho|.
\ee
Moreover, by the Cauchy-Schwarz inequality, $n[\varrho]$ belongs to $W^{1,1}(0,1)$ with periodic boundary conditions, and we have
$$\left\|\frac{dn[\varrho]}{dx}\right\|_{L^1}\leq 2\|\varrho\|_{\calJ_1}^{1/2}\left(\Tr \sqrt{H}|\varrho|\sqrt{H}\right)^{1/2}\leq C\|\varrho\|_{\calE}.$$

The energy space will be the following closed convex subspace of $\calE$:
$$\calE_+=\left\{\varrho\in \calE:\, \varrho\geq 0\right\}.$$
On $\calE_+$ we define the following free energy:
\be
\label{F}
F(\varrho)= \Tr \left(\varrho\log(\varrho)-\varrho\right)+\Tr (\sqrt{H}\varrho\sqrt{H}).
\ee
We will see in Section \ref{sect3} that $F$ is well-defined and continuous on $\calE_+$. If $\varrho\in \calE_+$, then the Cauchy-Schwarz inequality applied to \fref{srt2} gives 
$$\left|\frac{d}{dx}\sqrt{n[\varrho]}\right|\leq \left(\sum_{k=1}^\infty\rho_k\left|\frac {d\phi_k}{dx}\right|^2\right)^{1/2}.$$
Hence we have $\sqrt{n[\varrho]}\in \Hunper$ and, using \fref{srt1}, we get
\be
\label{sqrt}
\left\|\frac{d}{dx}\sqrt{n[\varrho]}\right\|_{L^2}\leq \left(\Tr \sqrt{H}\varrho\sqrt{H}\right)^{1/2}.
\ee
Recall also the following logarithmic Sobolev inequality for systems, proved in \cite{Dolbeault-Loss} and adapted to bounded domains in \cite{DFM}: for all $\varrho\in \calE_+$ we have
\be
\label{logsobo}
\Tr \varrho \log \varrho+\Tr (\sqrt{H}\varrho \sqrt{H})\geq \int_0^1n[\varrho]\log n[\varrho]dx+\frac{\log(4\pi)}{2}\Tr \varrho.
\ee
This inequality, coupled to \fref{sqrt} which gives $n[\varrho]\log n[\varrho]\in L^1(0,1)$, implies that $\Tr \varrho \log \varrho$ is bounded for all $\varrho\in \calE_+$.

Our main result is stated in the following theorem.
\begin{theorem}
\label{theo1}
Consider a density $n\in \Hunper$ such that $n>0$ on $[0,1]$. Then the following minimization problem with constraint:
\be
\label{min}
\min F(\varrho)\mbox{ for $\varrho\in \calE_+$ such that }n[\varrho]=n,
\ee
where $F$ is defined by \fref{F}, is attained for a unique density operator $\varrho[n]$, which has the following characterization. We have
\be
\label{char}
\varrho[n]=\exp\left(-(H+A)\right),
\ee
where $A$ belongs to the dual space $\H^{-1}_{per}$ of $\Hunper$ and the operator $H+A$ is taken in the sense of the associated quadratic form
\be
\label{qu}
Q_{A}(\varphi,\varphi)=\left\|\frac{d\varphi}{dx}\right\|_{L^2}^2+(A, |\varphi|^2)_{\H^{-1}_{per},\Hunper}.
\ee
\end{theorem}
{}From \fref{char}, it is possible to obtain a formula for $A$. Such formula is given in \fref{defAA}. The following remark shows that the functional space $\H^{-1}_{per}$ for the quantum chemical potential $A(x)$ is optimal. 
\begin{remark}
For a given $A\in \H^{-1}_{per}$, we prove further --see subsection \ref{subse}, Step 4 of the proof-- that the operator $H+A(x)$ (in the sense of quadratic forms) is self-adjoint and has a compact resolvent. Moreover, the associated quadratic form is a form-bounded perturbation of $u\mapsto \|u'\|_{L^2}^2$, so that $H+A(x)$ can be diagonalized on $L^2$ and its $k$-th eigenvalue $\rho_k$ has an asymptotic behaviour of the form $Ck^2$. Therefore, the $\H^1$ norm of the associated eigenvector $\phi_k$ is bounded by $Ck$. Consider now the operator $\varrho=\exp(-(H+A))$ and the associated density $n(x)$. By using the decay of the exponential, one can see that the series in \fref{srt2} is converging in $\H^1$. If we assume that $A$ belongs to the Sobolev space $H^s$, where $-1<s<0$, then by elliptic regularity one has $\phi_k\in \H^{s+2}$, and the series \fref{srt2} will converge in this Sobolev space, so we deduce that $n\in \H^{s+2}$. Hence, if $n$ belongs to $\H^1$ but does no
 t belong to any $\H^s$, $s>1$, then we have $A\in \H^{-1}_{per}$ and $A$ cannot be more regular.
\end{remark}
\begin{remark}
\label{remarque2}
The two main limitations of this theorem, the strict positivity of $n$ and the one-dimensional setting, are not essential for the first part of the theorem, the existence and uniqueness of the minimizer. This first result will be extended in a forthcoming work. However, these assumptions are essential in our proof of the second part of the theorem, the characterization of the minimizer. Indeed, the strict positivity of $n$ is crucial in Subsection \ref{subse}, Step 3, see e.g. Eq. \fref{char} and the argument after \fref{borneHmoinsun}. Moreover, the one-dimensional framework implies by Sobolev embeddings that $\Hunper$ is a Banach algebra, which enables to define the above quadratic form $Q_A$ in \fref{qu}.
\end{remark}

\ss
\ni
{\em Outline of the proof of Theorem \pref{theo1}.} The existence of the minimizer $\varrho[n]$ of the constrained problem \fref{min} is obtained by proving that minimizing sequences are compact and that the functional is lower semicontinuous. Compactness stems from uniform estimates that enable to apply Lemma \ref{strongconv}, whereas the lower semicontinuity comes from \fref{second} in Lemma \ref{strongconv} and from Lemma \ref{propentropie}. The uniqueness of the minimizer is a consequence of the strict convexity of the entropy (proved in Lemma \ref{propentropie}).

In order to characterize the minimizer $\varrho[n]$ of \fref{min}, we need to write the Euler-Lagrange equation for this minimization problem. This task is difficult because the constraint $n[\varrho]=n$ is not easy to handle when perturbing a density operator. We circumvent this difficulty by defining a new minimization problem with penalization, whose minimizer $\varrho_\eps$ will converge to $\varrho[n]$. Next, the Euler-Lagrange equation for the penalized problem reads
$$\sqrt{\varrho_\eps}\left(\log (\varrho_\eps)+H+A_\eps\right)\sqrt{\varrho_\eps}.$$
{}From this equation, in order to prove that 
$$\varrho_\eps=\exp\left(-(H+A_\eps)\right),$$
we show two important intermediate results, relying on the fact that $\varrho_\eps$ is a minimizer and on properties of the von Neumann entropy $\Tr(\varrho \log \varrho -\varrho)$. First, we prove that the kernel of $\varrho_\eps$ is reduced to $\{0\}$. Second, we prove that the family $(\phi_p^\eps)$ of eigenfunctions of $\varrho_\eps$, which is a Hilbert basis of $L^2(0,1)$, is in fact dense in $\Hunper$. This enables to prove that $(\phi_p^\eps)$ is the complete family of eigenfunctions of $H+A_\eps$, and to identify the associated eigenvalues. Finally, using the two assumptions discussed in Remark \ref{remarque2}, we are able to prove that $A_\eps$ converges in the $\H^{-1}_{per}$ strong topology, which is sufficient to pass to the limit as the penalization parameter $\eps$ goes to zero, and conclude the proof.

\section{Basic properties of the energy space and the entropy}
\label{sect3}
In this section, we prove a few basic results on the energy space $\calE_+$ that will be used in the paper.
\begin{lemma} \label{strongconv}
Let $(\varrho_k)_{k\in \NN}$ be  a bounded sequence of $\calE_+$. Then, up to an extraction of a subsequence, there exists $\varrho\in \calE_+$ such that
\be
\label{first}
\varrho_k\to\varrho\mbox{ in }\calJ_1\quad \mbox{and}\quad \sqrt{\varrho_k}\to \sqrt{\varrho}\mbox{ in }\calJ_2\quad \mbox{as } k\to +\infty
\ee
and
\be
\label{second}
\Tr(\sqrt{H}\varrho\sqrt{H})\leq \liminf_{k\to +\infty} \Tr(\sqrt{H}\varrho_k\sqrt{H}).
\ee
Furthermore, if one has
$$\Tr(\sqrt{H}\varrho\sqrt{H})= \lim_{k\to +\infty} \Tr(\sqrt{H}\varrho_k\sqrt{H})$$
then one can conclude in addition that
\be
\label{third}
\sqrt{H}\sqrt{\varrho_k}\to \sqrt{H}\sqrt{\varrho_k}\mbox{ in }\calJ_2\quad \mbox{as } k\to +\infty.
\ee
\end{lemma}
\begin{proof}
{\em Step 1: weak-$*$ convergence in $\calJ_1$.} First notice that the boundedness of $\varrho_k$ in $\calE_+$ implies by \fref{srt1} that the operator $\sqrt{H}\sqrt{\varrho_k}$ is bounded in the Hilbert space $\calJ_2$. Moreover, since $\varrho_k$ and $\sqrt{H}\varrho_k\sqrt{H}$ are positive and bounded in $\calJ_1$, we can extract subsequences such that $\varrho_k$ and $\sqrt{H}\varrho_k\sqrt{H}$ converge in the $\calJ_1$ weak-$*$ topology, that is, there exists two positive trace class operators $\varrho$, $A$ such that, for all compact operator $K\in \calK$, 
$$
\Tr (K \varrho_k)  \to \Tr (K \varrho)  \qquad ; \qquad \Tr (K \sqrt{H}\varrho_k\sqrt{H})  \to \Tr (K A).
$$
By application of Proposition 3.12 of \cite{Brezis}, we have 
\be
\label{second1}
\|A\|_{\calJ_1}\leq \liminf\|\sqrt{H}\varrho_k\sqrt{H}\|_{\calJ_1}=\liminf\Tr (\sqrt{H}\varrho_k\sqrt{H}).
\ee

\bs
\ni
{\em Step 2: Identification of $A$.} We show that $A=\sqrt{H}\varrho\sqrt{H}$. Indeed, let $K=(\sqrt{H}+I)^{-1}K'(\sqrt{H}+I)^{-1}$ with $K'$ compact. Using the cyclicity of trace with the bounded operators $\sqrt{\varrho_k} \sqrt{H}$ and $\sqrt{H}\sqrt{\varrho_k}$, we get
\begin{eqnarray*}
 \Tr (K \sqrt{H}\varrho_k\sqrt{H})&=& \Tr (\sqrt{\varrho_k} \sqrt{H} \,K \,\sqrt{H}\sqrt{\varrho_k}),\\
&=&\Tr (\sqrt{\varrho_k}\sqrt{H}  (\sqrt{H}+I)^{-1}K'(\sqrt{H}+I)^{-1}  \sqrt{H}\sqrt{\varrho_k})\\
&=&\Tr (\sqrt{\varrho_k} K'\sqrt{\varrho_k})-\Tr (\sqrt{\varrho_k} K'(\sqrt{H}+I)^{-1}\sqrt{\varrho_k})\\
&&-\Tr (\sqrt{\varrho_k} (\sqrt{H}+I)^{-1}K'\sqrt{\varrho_k})\\
&&+\Tr (\sqrt{\varrho_k} (\sqrt{H}+I)^{-1}K' (\sqrt{H}+I)^{-1}\sqrt{\varrho_k}),\\
&=&\Tr (K' \varrho_k)-\Tr (K'(\sqrt{H}+I)^{-1}\varrho_k)-\Tr ((\sqrt{H}+I)^{-1}K'\varrho_k)\\
&&+\Tr  ((\sqrt{H}+I)^{-1}K' (\sqrt{H}+I)^{-1}\varrho_k).
\end{eqnarray*}
Since $\varrho_k \to \varrho$ in the $\calJ_1$ weak-$*$ topology, and since $(\sqrt{H}+I)^{-1}K'(\sqrt{H}+I)^{-1}$, $K'(\sqrt{H}+I)^{-1}$, $(\sqrt{H}+I)^{-1}K'$ are compact operators, we have
\begin{eqnarray*}
 \Tr (K \sqrt{H}\varrho_k\sqrt{H})&\to& \Tr (K' \varrho)-\Tr (K'(\sqrt{H}+I)^{-1}\varrho)-\Tr ((\sqrt{H}+I)^{-1}K'\varrho)\\
&&+\Tr  ((\sqrt{H}+I)^{-1}K' (\sqrt{H}+I)^{-1}\varrho)\\
&=& \Tr (K' (\sqrt{H}+I)^{-1} A (\sqrt{H}+I)^{-1}).
\end{eqnarray*}
We thus obtain
\begin{eqnarray*}(\sqrt{H}+I)^{-1} A (\sqrt{H}+I)^{-1}&=&\varrho-(\sqrt{H}+I)^{-1}\varrho-\varrho (\sqrt{H}+I)^{-1}\\
&&+(\sqrt{H}+I)^{-1}\varrho(\sqrt{H}+I)^{-1}
\end{eqnarray*}
and it follows that $A=\sqrt{H} \rho \sqrt{H}$. In particular, \fref{second1} yields \fref{second}.

\bs
\ni
{\em Step 3: weak convergence in $\calJ_1$}. Let us prove now that $\varrho_k$  converges weakly in $\calJ_1$, that is, for all bounded operator $\sigma\in \calL(L^2(0,1))$,
$$
\Tr (\sigma \varrho_k)  \to \Tr (\sigma \varrho).
$$
We have
\begin{eqnarray*}
 \Tr (\sigma \varrho_k)&=& \Tr (\sqrt{\varrho_k} \, \sigma \, \sqrt{\varrho_k})\\
&=&\Tr (\sqrt{\varrho_k} (\sqrt{H}+I) (\sqrt{H}+I)^{-1} \sigma  (\sqrt{H}+I)^{-1} (\sqrt{H}+I)\sqrt{\varrho_k})\\
&=&\Tr((\sqrt{H}+I)^{-1} \sigma  (\sqrt{H}+I)^{-1} (\sqrt{H}+I)\varrho_k (\sqrt{H}+I))\\
&=&\Tr((\sqrt{H}+I)^{-1} \sigma  (\sqrt{H}+I)^{-1} \sqrt{H}\varrho_k \sqrt{H})+\Tr ((\sqrt{H}+I)^{-1} \sigma  \varrho_k)\\
&&+\Tr (\sigma  (\sqrt{H}+I)^{-1} \varrho_k)-\Tr((\sqrt{H}+I)^{-1} \sigma  (\sqrt{H}+I)^{-1}\varrho_k).
\end{eqnarray*}
Since $(\sqrt{H}+I)^{-1} \sigma  (\sqrt{H}+I)^{-1}$,  $(\sqrt{H}+I)^{-1} \sigma$ and $\sigma  (\sqrt{H}+I)^{-1}$ are compact, since $\varrho_k \to \varrho$ and $\sqrt{H}\varrho_k\sqrt{H} \to \sqrt{H}\varrho\sqrt{H}$ in $\calJ_1$ weakly-$*$, we can pass to the limit in the latter expression and obtain the weak convergence of $\varrho_k$.

\bs
\ni
{\em Step 4: strong convergence in $\calJ_1$}. To obtain the strong convergence in $\calJ_1$, it suffices now to apply a result of \cite{Simon-trace} that we recall here (specified to the case of $\calJ_1$). 
\begin{theorem}[Theorem 2.21 and addendum H of \cite{Simon-trace}]
\label{theoSimon}
Suppose that $A_k\to A$ weakly in the sense of operators and that $\|A_k\|_{\calJ_1}\to \|A\|_{\calJ_1}$. Then $\|A_k-A\|_{\calJ_1}\to 0$.
\end{theorem}
One can indeed apply this result since $\varrho_k$ converges weakly in $\calJ_1$ to $\varrho$ (which implies the weak operator convergence), with convergence of the respective norms:
$$\|\varrho_k\|_{\calJ_1}=\Tr \varrho_k\to \Tr \varrho=\|\varrho\|_{\calJ_1}.$$
This implies that the convergence of $\varrho_k$ in $\calJ_1$ is strong and the first part of the lemma is proved.

\bs
\ni
{\em Step 5: strong convergence of $\sqrt{\varrho_k}$ in $\calJ_2$}. We have
$$\|\varrho_k-\varrho\|_{\calL(L^2)}\leq \|\varrho_k-\varrho\|_{\calJ_1}.$$
Moreover, it is known that the norm convergence of $\varrho_k\geq 0$ to $\varrho\geq 0$ implies the norm convergence of $\sqrt{\varrho_k}$ to $\sqrt{\varrho}$ (see e.g. \cite{RS-80-I}). We claim that in fact we have 
\be
\label{preu}
\sqrt{\varrho_k}\to \sqrt{\varrho}\mbox{ in $\calJ_2$.}
\ee
To prove this fact, since $\calJ_2$ is a Hilbert space and since 
$$\|\sqrt{\varrho_k}\|^2_{\calJ_2}=\Tr \varrho_k\to \Tr \varrho = \|\varrho\|^2_{\calJ_2},$$
it suffices to prove that $\sqrt{\varrho_k}\rightharpoonup \sqrt{\varrho}$ in $\calJ_2$ weak. Let $\sigma\in \calJ_2$. One can choose a regularizing sequence $\sigma_k$ with finite rank such that 
\be
\label{ddd}
\sigma_k\to \sigma \mbox{ in }\calJ_2\mbox{ as }n\to +\infty.
\ee
For all $n,m\in \NN$, we have
\bee
|\Tr (\sqrt{\varrho_k}-\sqrt{\varrho})\sigma|&\leq & |\Tr (\sqrt{\varrho_k}-\sqrt{\varrho})\sigma_m|+|\Tr (\sqrt{\varrho_k}-\sqrt{\varrho})(\sigma-\sigma_m)|\\
&\leq& \|\sqrt{\varrho_k}-\sqrt{\varrho}\|_{\calL(L^2)}\|\sigma_m\|_{\calJ_1}+(\|\sqrt{\varrho_k}\|_{\calJ_2}+\|\sqrt{\varrho}\|_{\calJ_2})\|\sigma_m-\sigma\|_{\calJ_2}\\
&\leq&  \|\sqrt{\varrho_k}-\sqrt{\varrho}\|_{\calL(L^2)}\|\sigma_m\|_{\calJ_1}+C\|\sigma_m-\sigma\|_{\calJ_2},
\eee
which implies by \fref{ddd} that $\Tr (\sqrt{\varrho_k}-\sqrt{\varrho})\sigma\to 0$ as $n\to +\infty$. This means that $\sqrt{\varrho_k}\rightharpoonup \sqrt{\varrho}$ in $\calJ_2$ weak, which finally implies \fref{preu}.

\bs
\ni
{\em Step 6: strong convergence of $\sqrt{H}\sqrt{\varrho_k}$ in $\calJ_2$}. 
{}From now on, we assume that one has in addition the following convergence:
$$\|\sqrt{H}\sqrt{\varrho_k}\|^2_{\calJ_2}=\Tr (\sqrt{H}\varrho_k\sqrt{H})\to\Tr (\sqrt{H}\varrho\sqrt{H})=\|\sqrt{H}\sqrt{\varrho}\|^2_{\calJ_2}.$$
Consequently, if we prove that 
$$\sqrt{H}\sqrt{\varrho_k}\rightharpoonup \sqrt{H}\sqrt{\varrho}\mbox{ in }\calJ_2\mbox{ weak},$$ 
this weak convergence in the Hilbert space $\calJ_2$ will be in fact a strong convergence.

To this aim, we consider $\sigma\in \calJ_2$ and, for all $\eps\in(0,1)$, we decompose
$$
\begin{array}{l}
\ds \Tr (\sigma \sqrt{H}\sqrt{\varrho_k})=\Tr (\sigma (1+\eps \sqrt{H})^{-1}\sqrt{H}\sqrt{\varrho_k})+\Tr (\sigma (1-(1+\eps \sqrt{H})^{-1})\sqrt{H}\sqrt{\varrho_k})\\[2mm]
\ds \quad=\Tr (\sigma \sqrt{H}\sqrt{\varrho})+\Tr (\sigma_\eps\sqrt{H}(\sqrt{\varrho_k}-\sqrt{\varrho}))+\Tr ((\sigma_\eps-\sigma)\sqrt{H}\sqrt{\varrho})+\Tr ((\sigma-\sigma_\eps)\sqrt{H}\sqrt{\varrho_k})
\end{array}
$$
with $\sigma_\eps=\sigma (1+\eps \sqrt{H})^{-1}$.
For all $\eps>0$, the operator $(1+\eps \sqrt{H})^{-1}\sqrt{H}$ is bounded on $L^2(0,1)$, so $\sigma_\eps \sqrt{H}$ belongs to $\calJ_2$ and the previous step implies that
$$\lim_{k\to +\infty}\Tr (\sigma_\eps\sqrt{H}(\sqrt{\varrho_k}-\sqrt{\varrho}))=0.$$
Now we write
\bee
\left|\Tr ((\sigma_\eps-\sigma)\sqrt{H}\sqrt{\varrho})+\Tr ((\sigma-\sigma_\eps)\sqrt{H}\sqrt{\varrho_k})\right|
&\leq& \|\sigma-\sigma_\eps\|_{\calJ_2}\left(\|\sqrt{H}\sqrt{\varrho_k}\|_{\calJ_2}+\|\sqrt{H}\sqrt{\varrho}\|_{\calJ_2}\right)\\
&\leq& C\|\sigma-\sigma_\eps\|_{\calJ_2}.
\eee
Therefore, if we prove that $\sigma_\eps$ converges to $\sigma$ in $\calJ_2$ as $\eps\to 0$, we will have
$$\Tr (\sigma \sqrt{H}\sqrt{\varrho_k})\to \Tr (\sigma \sqrt{H}\sqrt{\varrho}) \quad \mbox{as }k\to +\infty$$
and the proof of the lemma will be complete.

Introduce the eigenfunctions and eigenvalues $(e_p,\mu_p)_{p\in \NN^*}$ of the operator $H$, which has a compact resolvent. We have
$$\|\sigma_\eps\|_{\calJ_2}^2=\sum_{p=1}^\infty\|\sigma_\eps e_p\|_{L^2}^2=\sum_{p=1}^\infty \frac{1}{(1+\eps\sqrt{\mu_p})^2}\|\sigma e_p\|_{L^2}^2$$
which converges to $\|\sigma\|_{\calJ_2}^2$ as $\eps \to 0$ by comparison theorem. Similarly, for all $\varphi\in L^2(0,1)$, we deduce from the convergence in $L^2$ of the series $\sum_p \varphi_pe_p$, where $\varphi_p=\int_0^1\varphi(x) e_p(x) dx$, that
$$(1+\eps\sqrt{H})^{-1} \varphi=\sum_{p=1}^\infty\frac{1}{1+\eps\sqrt{\mu_p}}\varphi_pe_p\to \varphi\mbox{ in }L^2(0,1)\mbox{ as }\eps \to 0.$$
Therefore, $\sigma_\eps$ and $\sigma_\eps^*=(1+\eps\sqrt{H})^{-1}\sigma$ converge strongly to $\sigma$ as $\eps\to 0$ and one can apply Gr\"umm's convergence theorem (see \cite{Simon-trace}, Theorem 2.19), which proves the convergence of $\sigma_\eps$ to $\sigma$ in $\calJ_2$. The proof of Lemma \ref{strongconv} is complete.
\end{proof}

\begin{lemma}
\label{propentropie}
The application $\varrho\mapsto \Tr (\varrho \log \varrho-\varrho)$ possesses the following properties.\\
(i) There exists a constant $C>0$ such that, for all $\varrho\in \calE_+$, we have
\be
\label{souslin}
\Tr (\varrho\log \varrho-\varrho)\geq -C\left(\Tr \sqrt{H}\varrho\sqrt{H}\right)^{1/2}.
\ee
(ii) Let $\varrho_k$ be a bounded sequence of $\calE_+$ such that $\varrho_k$ converges to $\varrho$ in $\calJ_1$, then $\varrho_k \log \varrho_k-\varrho_k$ converges to $\varrho \log \varrho-\varrho$ in $\calJ_1$.\\
(iii) The application $\varrho\mapsto \Tr (\varrho \log \varrho-\varrho)$ is strictly convex on $\calE_+$.
\end{lemma}
\begin{proof} 
{\em Step 1: proof of the inequality \fref{souslin}.} We shall use the following inequality, deduced from Lemma \ref{lieb} which is proven in the Appendix: there exists $C>0$ such that, for all $\varrho\in \calE_+$,
\be
\label{ineq2}
\sum_{p\geq 1} p^2\lambda_p[\varrho]\leq C\Tr \sqrt{H}\varrho \sqrt{H},
\ee
where we have denoted by $(\lambda_p[\varrho])_{p\geq 1}$ the nonincreasing sequence of nonzero eigenvalues of $\varrho$ (this sequence is finite or infinite). The function $s\mapsto \beta(s)=s\log s-s$ is negative on $[0,e]$ and positive increasing on $[e,+\infty)$. Let
$$C_1=\sup_{s\in [0,e]}\frac{|s\log s-s|}{\sqrt{s}}<+\infty.$$
Let $\varrho\in \calE_+$ and denote by $(\lambda_p[\varrho])_{p> p_0}$ the eigenvalues of $\varrho$ that belong to the interval $(0,e]$. We have
\bee
-\Tr \beta(\varrho)\leq \sum_{p> p_0}\left|\beta(\lambda_p[\varrho])\right|&\leq& C_1\sum_{p> p_0}\sqrt{\lambda_p[\varrho]}\\
&\leq& C_1\left(\sum_{p> p_0}p^2\lambda_p[\varrho]\right)^{1/2}\left(\sum_{p> p_0}\frac{1}{p^2}\right)^{1/2}\\
&\leq &\frac{C_1}{\sqrt{p_0}}\left(\Tr \sqrt{H}\varrho \sqrt{H}\right)^{1/2},
\eee
which proves \fref{souslin}.

\bs
\ni
{\em Step 2: proof of (ii).} 
Consider a sequence $\varrho_k$ bounded in $\calE_+$, such that $\varrho_k\to \varrho$ in $\calJ_1$. Let $M=\sup_{k}\|\varrho_k\|_{\calL(L^2)}<+\infty$. There exists a constant $C_M>0$ such that
$$\forall s\in [0,M],\quad \left|s\log s-s\right|\leq C_Ms^{3/4}.$$
Thus, for all $\eps>0$, denoting again $\beta(s)=s\log s-s$, we get
\bee
\sum_{\lambda_p[\varrho_k]\leq \eps}\left|\beta(\lambda_p[\varrho_k])\right|&\leq& C_M\sum_{\lambda_p[\varrho_k]\leq \eps}(\lambda_p[\varrho_k])^{3/4}\leq C_M\eps^{1/4}\sum_{\lambda_p[\varrho_k]\leq \eps}(\lambda_p[\varrho_k])^{1/2}\\
&\leq& C_M\eps^{1/4}\left(\sum_{p\geq 1}p^2\lambda_p[\varrho_k]\right)^{1/2}\left(\sum_{p\geq 1}\frac{1}{p^2}\right)^{1/2}\\
&\leq &C\eps^{1/4}\left(\Tr \sqrt{H}\varrho_k \sqrt{H}\right)^{1/2}\leq C\eps^{1/4},
\eee
where $C$ is independent of $k$ and where we used \fref{ineq2}. The same inequality holds for the limit $\varrho$. Let $\eps>0$ and let us decompose
$$\beta(s)=\beta_1(s)+\beta_2(s)=(\beta\un_{s\leq \eps})(s)+(\beta\un_{s> \eps})(s).$$
For all $\eps$, one has
\bee
\Tr|\beta(\varrho_k)-\beta(\varrho)|
&\leq &\Tr|\beta_1(\varrho_k)|+\Tr|\beta_1(\varrho)|+\Tr|\beta_2(\varrho_k)-\beta_2(\varrho)|\\
&\leq&C\eps^{1/4}+\Tr|\beta_2(\varrho_k)-\beta_2(\varrho)|,
\eee
so that the result is proved if we show that $\beta_2(\varrho_k)$ converges to $\beta_2(\varrho)$ strongly in $\calJ_1$. According to Theorem \ref{theoSimon}, it is enough to prove that $\beta_2(\varrho_k)$ converges weakly to $\beta_2(\varrho)$ in $\calJ_1$ and that $\|\beta_2(\varrho_k)\|_{\calJ_1} \to \|\beta_2(\varrho)\|_{\calJ_1}$ to obtain the strong convergence in $\calJ_1$. We prove first the weak convergence. To this aim, we choose $\eps$ such that $\lambda_p[\varrho]\neq \eps$ for all $p\in \NN^*$ and denote
$$N=\max \left\{p:\,\lambda_p[\varrho] > \eps\right\}.$$
According to Lemma \ref{convFP}, we have
\be \label{cvvalprop}
\lambda_p[\varrho_k] \to \lambda_p[\varrho], \quad \forall p \geq 1,
\ee 
and we can choose $k$ large enough so that we have 
$$\lambda_p[\varrho_k]>\eps\mbox{ for all } p\leq N\mbox{ and }\lambda_p[\varrho_k]<\eps \mbox{ for all } p> N.$$
Besides, following again Lemma \ref{convFP}, we choose some eigenbasis $(\phi_p^k)_{p\in\NN^*}$ and $(\phi_p)_{p\in\NN^*}$ of $\varrho_k$ and $\varrho$, respectively, such that
\be
\label{cvvectprop}
\forall p\in \NN^*,\quad \lim_{k\to \infty}\|\phi_p^k-\phi_p\|_{L^2}=0.
\ee
Then, the actions of $\beta_2(\varrho_k)$ and $\beta_2(\varrho)$ on any $\varphi \in L^2(0,1)$ read
$$
\beta_2(\varrho_k) \varphi=\sum_{p=1}^N \beta(\lambda_p[\varrho_k]) (\phi_p^k, \varphi) \phi_p^k \qquad ; \qquad \beta_2(\varrho) \varphi=\sum_{p=1}^N \beta(\lambda_p[\varrho]) (\phi_p, \varphi) \phi_p,
$$
where $(\cdot,\cdot)$ denotes the $L^2(0,1)$ scalar product (taken linear with respect to its second variable and anti-linear with respect to its first variable). Therefore, for any bounded operator $B$,
$$
\Tr \left(\beta_2(\varrho_k) B\right)= \sum_{p=1}^N \beta(\lambda_p[\varrho_k]) (\phi_p^k, B \phi_p^k) \to  \sum_{p=1}^N \beta(\lambda_p[\varrho]) (\phi_p, B \phi_p)=\Tr \left(\beta_2(\varrho) B\right),
$$
thanks to \fref{cvvalprop}, \fref{cvvectprop} and the continuity of the function $\beta$. This proves the weak convergence of $\beta_2(\varrho_k)$ in $\calJ_1$. Regarding the convergence of the norm, we have directly
$$
\|\beta_2(\varrho_k)\|_{\calJ_1}= \sum_{p=1}^N |\beta(\lambda_p[\varrho_k])| \to \sum_{p=1}^N |\beta(\lambda_p[\varrho])|=\|\beta_2(\varrho)\|_{\calJ_1},
$$
and item (ii) is proved.

\bs
\ni
{\em Step 3: proof of the strict convexity (iii).}  We recall first the Peierls inequality \cite{Simon-trace}: let $(u_i)_{i \geq 1}$ be an orthonormal basis of $L^2(0,1)$, whose scalar product is denoted by $(\cdot,\cdot)$; then, setting $\beta(s)=s\log s-s$, we have
\be
\label{peierls}
\sum_{i \geq 1} \beta ((u_i,\varrho \,u_i)) \leq \Tr \left(\beta(\varrho) \right).
\ee
Indeed, denoting by $(\lambda_i, \phi_i)_{i \geq 1}$ the spectral elements of $\varrho \in \calE_+$, we have
$$
(u_i,\beta(\varrho)u_i)=\sum_{j \geq 1} \beta(\lambda_j) |(\phi_j,u_i)|^2.
$$
Since $\sum_{j \geq 1} |(\phi_j,u_i)|^2=1$, it follows from the Jensen inequality that
$$
\beta\left((u_i,\varrho \, u_i) \right) = \beta\left( \sum_{j \geq 1}\lambda_j |(\phi_j,u_i)|^2 \right) \leq \sum_{j \geq 1} \beta(\lambda_j) |(\phi_j,u_i)|^2.
$$
Summing up the latter relation with respect to $i$ and using the relation  $\sum_{i \geq 1} |(\phi_j,u_i)|^2=1$, the Peierls inequality \fref{peierls} follows. 

Consider now $\varrho_1$, $\varrho_2$ in $\calE_+$ such that  $\varrho_1 \neq \varrho_2$.  Let $t\in (0,1)$ and denote by  $(\mu_i, \psi_i)_{i \in \NN^*}$ the spectral elements of the operator $t \varrho_1 +(1-t) \varrho_2$. Then
$$
\Tr \left(\beta(t \varrho_1 +(1-t) \varrho_2 ) \right)=\sum_{i \geq 1} \beta(\mu_i)=\sum_{i \geq 1} \beta((\psi_i,(t \varrho_1 +(1-t) \varrho_2)\psi_i)).
$$
There exists at least one index $i_0$ such that $(\psi_{i_0}, \varrho_1 \psi_{i_0})\neq (\psi_{i_0},\varrho_2 \psi_{i_0})$. Indeed, if not, we would have $\varrho_1=\varrho_2$ since $(\psi_i)_{i \in \NN^*}$ is an orthonormal basis of $L^2(0,1)$.
Since $\beta$ is strictly convex, it thus comes,
$$
\sum_{i \geq 1} \beta(t (\psi_i,\varrho_1\,\psi_i)+(1-t)(\psi_i,\varrho_2\,\psi_i)) < \sum_{i \geq 1} \left[t \beta((\psi_i,\varrho_1\,\psi_i))
+(1-t)\beta((\psi_i,\varrho_2\,\psi_i))\right].
$$
Using the Peierls inequality \fref{peierls} to control the right hand side, it comes finally
$$\Tr \left(\beta(t \varrho_1 +(1-t) \varrho_2 ) \right)< t \Tr \left(\beta(\varrho_1) \right)+(1-t)\Tr \left(\beta(\varrho_2 )\right),$$
which yields the strict convexity of the functional.
\end{proof}

\section{Existence and uniqueness of the minimizer}
\label{sect4}
In this section, we prove the first part of our main Theorem \ref{theo1}. More precisely, we prove the following proposition.
\begin{proposition}
\label{propmin2}
Consider a density $n(x)$ such that $n>0$ on $[0,1]$ and $n\in \Hunper$. Then the minimization problem with constraint
\be
\label{min2}
\min F(\varrho)\mbox{ for $\varrho\in \calE_+$ such that }n[\varrho]=n,
\ee
where $F$ is defined by \fref{F}, is attained for a unique density operator $\varrho[n]$.
\end{proposition}
\begin{proof}
We denote
$$\calA=\left\{\varrho\in \calE_+ \mbox{ such that }n[\varrho]=n\right\}.$$

\bs
\ni
{\em Step 1: $\calA$ is not empty.} We start with a simple, but fundamental remark: thanks to our assumption on the density $n(x)$, the set $\calA$ is not empty. Indeed, let $\phi_1:=\|n\|^{-1/2}_{L^1} \sqrt{n}$ and complete $\phi_1$ to an orthonormal basis  $(\phi_i)_{i \geq 1}$ of $L^2(0,1)$. The function $n$ belongs to $\Hunper$. Hence, by Sobolev embedding in dimension one, $n$ is continuous and, from $n>0$, we deduce that 
$$n(x)\geq \min_{[0,1]}n>0$$
and then $\sqrt{n}\in \Hunper$. For all $\psi \in L^2(0,1)$, consider the density operator $\nu$ defined by
\be
\label{nu}
\nu \psi:=\sqrt{n}\, (\sqrt{n},\psi),
\ee
we find
\bee
\Tr (\sqrt{H} \nu \sqrt{H})&=& \|\sqrt{H}\sqrt{\nu}\|^2_{\calJ_2}= \sum_{i \geq 1}(\sqrt{H}\sqrt{\nu} \phi_i,\sqrt{H}\sqrt{\nu}\phi_i)\\
&=& (\sqrt{H}\sqrt{n},\sqrt{H}\sqrt{n})=\int_0^1 \left| \frac{d}{dx}\sqrt{n}\right|^2 dx < \infty,\\
\Tr (\Phi \nu) &=&\int_0^1 n(x)\Phi(x)dx\qquad \forall \Phi\in L^\infty(0,1),
\eee
so, by the characterization \fref{weak}, $\nu$ belongs to $\calA$.

\bs
\ni
{\em Step 2: $F$ is bounded from below on $\calA$}. From \fref{logsobo}, we deduce that, for all $\varrho\in \calA$,
\be
\label{logsobo2}
F(\varrho)\geq \int_0^1n(x)\log n(x)dx+\left(\frac{\log(4\pi)}{2}-1\right)\int_0^1n(x)dx>-\infty,
\ee
since by Sobolev embedding $n$ is bounded. Therefore, one can consider a minimizing sequence $(\varrho_k)_{k\in \NN}$ for \fref{min2}, i.e. a sequence $\varrho_k\in \calA$ such that
$$\lim_{k\to +\infty}F(\varrho_k)=\inf_{\sigma\in\calA}F(\sigma)>-\infty.$$

\bs
\ni
{\em Step 3: uniform bound in $\calE$.} Let us prove that $(\varrho_k)_{k\in \NN}$ is a bounded sequence of $\calE_+$. 
Since $\varrho_k\in \calA$, we already have
$$\|\varrho_k\|_{\calJ_1}=\Tr \varrho_k=\int_0^1n(x)dx<+\infty.$$
Moreover, since the density operator $\nu$ defined by \fref{nu} belongs to $\calA$, we have, for $k$ large enough,
$$\Tr (\varrho_k\log \varrho_k-\varrho_k)+\Tr (\sqrt{H}\varrho_k\sqrt{H})=F(\varrho_k)\leq F(\nu)+1<+\infty.$$
Hence, using the inequality \fref{souslin}, we obtain
$$-C\left(\Tr \sqrt{H}\varrho_k\sqrt{H}\right)^{1/2}+\Tr (\sqrt{H}\varrho_k\sqrt{H})\leq F(\nu)+1<+\infty,$$
thus
$$\sup_{k\in \NN}\Tr (\sqrt{H}\varrho_k\sqrt{H})<+\infty.$$

\bs
\ni
{\em Step 4: convergence to the minimizer.} Since $(\varrho_k)_{k\in \NN}$ is a bounded sequence of $\calE_+$, one can apply Lemma \ref{strongconv} to deduce that, after extraction of a subsequence, we have
\be
\label{con1}
\varrho_k\to\varrho\mbox{ in }\calJ_1\mbox{ as }k\to +\infty
\ee
and
\be
\label{con2}
\Tr(\sqrt{H}\varrho\sqrt{H})\leq \liminf_{k\to +\infty} \Tr(\sqrt{H}\varrho_k\sqrt{H}).
\ee
Next, by \fref{con1} and Lemma \ref{propentropie} {\em (ii)}, we get
$$\Tr (\varrho_k\log \varrho_k-\varrho_k)\to \Tr (\varrho\log \varrho-\varrho)\mbox{ as }k\to +\infty,$$
which yields, with \fref{con2},
\be
\label{eqmin}
F(\varrho)\leq \liminf_{k\to+\infty} F(\varrho_k)=\inf_{\sigma\in\calA}F(\sigma).
\ee
Let $\Phi\in L^\infty(0,1)$ and denote also by $\Phi$ the bounded multiplication operator by $\Phi$. Since $\varrho_k$ converges to $\varrho$ in $\calJ_1$, we have
$$\int_0^1\Phi(x)n(x)dx=\Tr (\Phi \varrho_k)\to \Tr (\Phi \varrho) \mbox{ as }k\to +\infty,$$
thus, from the characterization \fref{weak}, we deduce that $n[\varrho]=n$, which means that $\varrho\in \calA$.
This enables finally to conclude from \fref{eqmin} that, in fact, we have the equality
$$F(\varrho)=\inf_{\sigma\in\calA}F(\sigma)=\min_{\sigma\in\calA}F(\sigma).$$
The uniqueness of the minimizer is a consequence of the strict convexity of $F$, see Item {\em (iii)} of Lemma \ref{propentropie}. 
\end{proof}

\section{Characterization of the minimizer}
\label{sect5}

This section is devoted to the second part of our main Theorem \ref{theo1}, the characterization of the minimizer. As we explained  at the end of Section \ref{sect2}, we need to define a penalized version of our minimization problem.

\subsection{A penalized minimization problem}

Consider a density $n(x)$ such that $n>0$ on $[0,1]$ and $n\in \Hunper$. For all $\eps\in (0,1]$ we define the penalized free energy functional, for all $\varrho\in \calE_+$:
$$
F_\eps(\varrho)= \Tr (\varrho \log \varrho - \varrho)+\Tr (\sqrt{H}\varrho\sqrt{H})+\frac{1}{2 \eps} \| n[\varrho]-n\|^2_{L^2}. 
$$
\begin{proposition}
\label{lpropminreg}
Let $\eps\in (0,1)$ and let $n\in \Hunper$ such that $n>0$ on $[0,1]$. The minimization problem without constraint
\be
\label{minreg}
\min_{\varrho\in \calE_+}F_\eps(\varrho)
\ee
where $F_\eps$ is defined above, is attained for a unique density operator $\varrho_\eps[n]$, which has the following characterization: we have
\be
\label{charreg}
\varrho_\eps[n]=\exp \left(-(H+A_\eps)\right).
\ee
where $A_\eps\in \Hunper$.
\end{proposition}
\begin{proof} Since the entropy functional $\Tr (\varrho \log \varrho - \varrho)$ is not differentiable on $\calE_+$, we regularize it. For all $\eta \in [0,1]$ and $s\in \RR_+$, we define the regularized entropy
$$\beta_\eta(s)=(s+\eta) \log (s+\eta)-s-\eta\log \eta,$$
and the associated free energy functional, for all $\varrho\in \calE_+$:
$$
F_{\eps,\eta}(\varrho)= \Tr \left(\beta_\eta(\varrho)\right)+\Tr (\sqrt{H}\varrho\sqrt{H})+\frac{1}{2 \eps} \| n[\varrho]-n\|^2_{L^2}. 
$$
Notice that $\beta'_\eta(s)=\log(s+\eta)$, $\beta_\eta(0)=0$, and that $\beta_\eta$ is strictly convex on $\RR_+$ and holomorphic on $(-\eta, \infty) \times \RR$ for the convenient branch.

\bs
\ni
{\em Step 1: minimization of $F_{\eps,\eta}$.}  In this step, we prove that for all $\eta\in [0,1]$, the problem
\be
\label{minregeta}
\min_{\varrho\in \calE_+}F_{\eps,\eta}(\varrho)
\ee
admits a unique minimizer $\varrho_{\eps,\eta}$. Notice that for $\eta=0$, this problem is nothing but \fref{minreg}: in the statement of the Proposition, we have denoted shortly $\varrho_\eps=\varrho_{\eps,0}$.

By \fref{sqrt} and a Sobolev embedding in dimension one, we have
$$\|n[\varrho]\|_{L^\infty}\leq C\Tr \varrho+C\Tr (\sqrt{H}\varrho\sqrt{H}),$$
so the functional $F_\eps^\eta$ is well-defined on $\calE_+$ for all $\eta\in [0,1]$ and $\eps\in (0,1]$.
We will need the following technical lemma on the function $\beta_\eta$.
\begin{lemma}
\label{propentropieeps}
The application $\varrho\mapsto \beta_\eta (\varrho)$ possesses the following properties.\\
(i) There exists a constant $C>0$ such that, for all $\varrho\in \calE_+$ and for all $\eta\in [0,1]$, we have
\be
\label{souslineps}
\Tr \beta_\eta (\varrho)\geq -C\left(\Tr \sqrt{H}\varrho\sqrt{H}\right)^{1/2}.
\ee
(ii) Let $\varrho_k$ be a bounded sequence of $\calE_+$ such that $\varrho_k$ converges to $\varrho$ in $\calJ_1$, then for all $\eta\in [0,1]$, $\beta_\eta(\varrho_k)$ converges to $\beta_\eta (\varrho)$ in $\calJ_1$.\\
(iii) For all $\eta\in [0,1]$, the application $\varrho\mapsto \Tr \beta_\eta (\varrho)$ is strictly convex on $\calE_+$.\\
(iv) Consider a sequence $\varrho_\eta$ bounded in $\calE_+$ such that $\varrho_\eta\to \varrho$ in $\calJ_1$ as $\eta\to 0$. Then $\Tr \beta_\eta(\varrho_\eta)$ converges to $\Tr \beta_0 (\varrho)$ as $\eta\to 0$.
\end{lemma}
\ni
{\em Proof of the lemma.}
It is not difficult to adapt the proof of Lemma \ref{propentropie} in order to show Items {\em (i), (ii), (iii)}. We shall only prove Item {\em (iv)}, proceeding similarly to Step 2 of Lemma \ref{propentropie}. We first notice that the function $\beta_\eta$ converges to $\beta_0$ uniformly on all $[0,M]$, $M>0$, and that one has
$$\forall s\in [0,M],\quad \left|\beta_\eta(s)\right|\leq C_M\sqrt{s},$$
with $C_M$ independent of $\eta$. Let $M=\sup_\eta \|\varrho_\eta\|_{\calL(L^2)}<+\infty$. For all $N\in \NN^*$, by using the inequality \fref{ineq2}, we get
$$
\sum_{p\geq N}\left|\beta_\eta(\lambda_p[\varrho_\eta])\right|\leq C_M\sum_{p\geq N}\sqrt{\lambda_p[\varrho_\eta]}\leq \frac{C_M}{\sqrt{N}}\left(\Tr \sqrt{H}\varrho_\eta \sqrt{H}\right)^{1/2}\leq \frac{C_M}{\sqrt{N}},
$$
where we used the fact that $(\varrho_k)$ is a bounded sequence of $\calE_+$ and where $\lambda_p[\varrho_\eta]$ denotes the $p$-th nonzero eigenvalue of $\varrho$.
Hence, decomposing
\bee
\left|\Tr \beta_\eta(\varrho_\eta) -\Tr \beta_0(\varrho)\right|&\leq&
\sum_{p< N}\left|\beta_\eta(\lambda_p[\varrho_\eta])-\beta_0(\lambda_p[\varrho])\right|\\
&&+\sum_{p\geq N}\left|\beta_\eta(\lambda_p[\varrho_\eta])\right|+\sum_{p\geq N}\left|\beta_0(\lambda_p[\varrho])\right|,
\eee
one deduces from the uniform convergence of $\beta_\eta$ to $\beta_0$ and from
$$
|\lambda_p[\varrho_\eta]-\lambda_p[\varrho]| \leq \|\varrho_\eta-\varrho \|_{\calL(L^2)}\leq \|\varrho_\eta-\varrho\|_{\calJ_1}, \qquad \forall p \geq 1
$$
(see the proof of Lemma \ref{convFP} in the Appendix) that
$$\left|\Tr \beta_\eta(\varrho_\eta) -\Tr \beta_0(\varrho)\right|\to 0 \mbox{ as }\eta\to 0.$$
The proof of Lemma \ref{propentropieeps} is complete.
\qed

Let us now study the minimization problem \fref{minregeta}, for fixed $\eta\in [0,1]$. For all $\varrho\in \calE_+$, we deduce from \fref{souslineps} that
$$F_{\eps,\eta}(\varrho)\geq -C\left(\Tr \sqrt{H}\varrho\sqrt{H}\right)^{1/2}+\Tr (\sqrt{H}\varrho\sqrt{H})+\frac{1}{4 \eps} \| n[\varrho]\|^2_{L^2}-\frac{1}{2 \eps} \|n\|^2_{L^2},$$
where we used $(a-b)^2\geq \frac{1}{2}a^2-b^2$. From this inequality, we deduce two facts. First, that $\inf_{\varrho\in \calE_+}F_{\eps,\eta}(\varrho)>-\infty$. Second, that any minimizing sequence $\varrho_k$ is bounded in $\calE_+$. Indeed, we have
$$F_{\eps,\eta}(\varrho_k)\leq F_{\eps,\eta}(0)= \frac{1}{2 \eps} \|n\|^2_{L^2},$$
thus
\bee
&&-C\left(\Tr \sqrt{H}\varrho_k\sqrt{H}\right)^{1/2}+\Tr( \sqrt{H}\varrho_k\sqrt{H})+\frac{1}{4 \eps} \sup_{k}\| n[\varrho_k]\|_{L^1}^2\\
&&\hspace{4cm}\leq \sup_{k}F_{\eps,\eta}(\varrho_k)+\frac{1}{2 \eps} \|n\|^2_{L^2}\leq \frac{1}{\eps} \|n\|^2_{L^2}.
\eee
This implies that, for all $k$,
\be
\label{hammamet}
\Tr \varrho_k+\Tr \sqrt{H}\varrho_k\sqrt{H}\leq C_\eps
\ee
where $C_\eps$ is a positive constant independent of $k$ and $\eta$.
 Hence, according to Lemma \ref{strongconv}, one can extract a subsequence still denoted $\varrho_k$ such that the convergences \be
\label{con2bis}
\varrho_k\to\varrho\mbox{ in }\calJ_1, \qquad \Tr(\sqrt{H}\varrho\sqrt{H})\leq \liminf_{k\to +\infty} \Tr(\sqrt{H}\varrho_k\sqrt{H})
\ee
 hold true as $k\to +\infty$. By Lemma \ref{propentropieeps} {\em (ii)}, we have
\be
\label{con3}
\Tr \beta_\eta(\varrho_k)\to \Tr\beta_\eta(\varrho).
\ee 

Let us now prove that $n[\varrho_k]$ converges to $n[\varrho]$ in $L^\infty(0,1)$, which implies in particular that
\be
\label{con4}
 \| n[\varrho_k]-n\|^2_{L^2}\to  \| n[\varrho]-n\|^2_{L^2}.
 \ee
We have
$$\|\sqrt{n[\varrho_k]}\|_{\H^1}^2=\int_0^1n[\varrho_k](x)dx+\left\|\frac{d}{dx}\sqrt{n[\varrho_k]}\right\|_{L^2}^2\leq \Tr\varrho_k+\Tr \sqrt{H}\varrho_k\sqrt{H}<+\infty,$$
where we used \fref{sqrt}. Therefore, the sequence $(\sqrt{n[\varrho_k]})_{k\in \NN}$ is bounded in $\H^1(0,1)$, and by Sobolev embedding one can extract a subsequence such that $\sqrt{n[\varrho_k]}$ converges to a function $f\in \calC^0([0,1])$ in the  $L^\infty(0,1)$ topology. This implies that $n[\varrho_k]$ converges to $f^2$ in $L^\infty$.

Moreover, for all $\Phi\in L^\infty(0,1)$, we know from \fref{con2bis} that
$$\int_0^1n[\varrho_k]\Phi dx=\Tr \varrho_k \Phi\to \Tr \varrho \Phi=\int_0^1n[\varrho]\Phi dx,$$
which means that $n[\varrho_k]$ converges weakly to $n[\varrho]$ in $L^1(0,1)$. This enables to identify the limit: we have in fact $f^2=n[\varrho]$.

Finally, \fref{con2bis}, \fref{con3} and \fref{con4} yield
$$F_{\eps,\eta}(\varrho)\leq \liminf F_{\eps,\eta} (\varrho_k)=\inf_{\sigma\in \calE_+}F_{\eps,\eta}(\sigma),$$
so $\varrho\in \calE_+$ is a minimizer of \fref{minregeta}. Furthermore, one remarks that the application $\varrho\mapsto n[\varrho]$ is linear, so the application
$$\varrho\mapsto \| n[\varrho_k]-n\|^2_{L^2}$$
is convex, and it can be deduced from Lemma \ref{propentropieeps} {\em (iii)} that $F_\eps^\eta$ is strictly convex: the minimizer $\varrho_{\eps,\eta}[n]$ is unique. In the sequel of this proof, $n$ being fixed, we denote shortly $\varrho_{\eps,\eta}$ instead of $\varrho_{\eps,\eta}[n]$. Notice that, from \fref{hammamet}, one gets an estimate independent of the parameter $\eta \in [0,1]$: for all $\eps \in (0,1]$, one has
\be
\label{estimeta}
\sup_{\eta\in [0,1]} \Tr \varrho_{\eps,\eta}+\sup_{\eta\in [0,1]}\Tr \sqrt{H}\varrho_{\eps,\eta}\sqrt{H}<+\infty.
\ee

\bs
\ni
{\em Step 2: differentiation of $F_\eps^\eta$ for $\eta>0$.}  We will use the following lemma, whose proof is in the Appendix.
\begin{lemma} \label{lemdiff} Let $\eta\in (0,1]$. Let $\varrho \in \calE_+$ and let $\omega$ be a trace-class self-adjoint operator. Then, the G\^ateaux derivative of the application 
$$\varrho\mapsto \widetilde F_\eta(\varrho)=\Tr \beta_\eta(\varrho)$$ 
at $\varrho$ in the direction $\omega$ is well-defined and we have
$$
D\widetilde F_\eta(\varrho) (\omega) = \Tr \left(\beta_\eta'( \varrho) \omega\right). 
$$ 
\end{lemma}
Let $h$ be a bounded Hermitian operator. For $\varrho\in \calE_+$, consider the operator $\varrho+t \sqrt{\varrho}\, h \sqrt{\varrho}$. Assume that $t \in [-t_0,t_0]$, with $0<t_0 \| h \| \leq 1$. For such values of $t$ and for all $\varphi \in L^2(0,1)$, we have
\begin{eqnarray*}
(\varrho \varphi,\varphi)+t (\sqrt{\varrho}\, h \sqrt{\varrho} \varphi,\varphi)&=&\|\sqrt{\varrho} \varphi\|_{L^2(\Omega)}^2+
t(h \sqrt{\varrho} \varphi,\sqrt{\varrho }\varphi),\\
& \geq& \|\sqrt{\varrho} \varphi\|_{L^2(\Omega)}^2 (1-|t| \|h\|),\\
 &\geq& 0.
\end{eqnarray*}
Therefore $\varrho+t \sqrt{\varrho}\, h \sqrt{\varrho}$ is nonnegative, self-adjoint and belongs to $\calE_+$ since
$$
 \left| \Tr(\sqrt{H} \sqrt{\varrho} \,h \sqrt{\varrho} \sqrt{H}) \right| \leq \|h\| \|\sqrt{H} \sqrt{\varrho}\|_{\calJ_2}^2 =\|h\| \Tr (\sqrt{H} \varrho \sqrt{H}) < \infty. 
$$
Moreover, we have the following estimates:
$$\|\sqrt{\varrho}\,h\|^2_{\calJ_2}\leq \|h\|^2\|\varrho\|_{\calJ_1},\quad \|\sqrt{H} \sqrt{\varrho}\,h\|^2_{\calJ_2}\leq \|h\|^2 \Tr (\sqrt{H}\varrho \sqrt{H}),$$
$$\Tr (\sqrt{\varrho}\,|h| \sqrt{\varrho})\leq \|h\|\|\varrho\|_{\calJ_1},\quad \Tr (\sqrt{H}\sqrt{\varrho}\,|h| \sqrt{\varrho}\sqrt{H})\leq \|h\| \Tr (\sqrt{H}\varrho \sqrt{H}).$$
Therefore, by linearity, the following equality holds in $W^{1,1}_{per}(0,1)\subset L^\infty(0,1)\subset L^2(0,1)$:
$$n[\varrho+t\sqrt{\varrho}\,h\sqrt{\varrho}]=n[\varrho]+tn[\sqrt{\varrho}\,h\sqrt{\varrho}],$$
which yields, for all $t\neq 0$;
\be
\label{derivn}
\frac{\|n[\varrho+t\sqrt{\varrho}\,h\sqrt{\varrho}]-n\|_{L^2}^2-\|n[\varrho]-n\|_{L^2}^2}{2t}=\int_0^1n[\sqrt{\varrho}\,h\sqrt{\varrho}](x)\left(n[\varrho]-n\right)(x)dx+\calO(t).
\ee
{}From Lemma \ref{lemdiff} and from \fref{derivn}, one deduces the following expression for the G\^ateaux derivative of $F_{\eps,\eta}$ in the direction $ \omega=\sqrt{\varrho_{\eps,\eta}}\,h\sqrt{\varrho_{\eps,\eta}} $:
\be
\begin{array}{l}
\ds \lim_{t\to 0}\frac{F_{\eps,\eta}(\varrho_{\eps,\eta}+t\sqrt{\varrho _{\eps,\eta}}\,h\sqrt{\varrho_{\eps,\eta}})-F_{\eps,\eta}(\varrho_{\eps,\eta})}{t}\\[4mm]
\ds \qquad =\Tr \left(\beta'_\eta(\varrho_{\eps,\eta}) \sqrt{\varrho_{\eps,\eta}}\,h \sqrt{\varrho_{\eps,\eta}} \right)+ \Tr (\sqrt{H} \sqrt{\varrho_{\eps,\eta}}\, h \sqrt{\varrho_{\eps,\eta}} \sqrt{H})+\Tr (A_{\eps,\eta} \sqrt{\varrho_{\eps,\eta}}\,h \sqrt{\varrho_{\eps,\eta}})\\[2mm]
\ds \qquad =\Tr \left(\sqrt{\varrho_{\eps,\eta}}\left(\beta'_\eta(\varrho_{\eps,\eta}) + H +A_{\eps,\eta} \right)\sqrt{\varrho_{\eps,\eta}}\,h\right), 
\end{array}
\label{exprDF}
\ee
where we used the cyclicity of the trace and where we have denoted
\be
\label{Aeps}
A_{\eps,\eta}(x)=\frac{1}{\eps} \left(n[\varrho_{\eps,\eta}]-n\right)(x).
\ee
Note that $A_{\eps,\eta}$ denotes here, with an abuse of notation, either the $L^\infty$ function $A_{\eps,\eta}$, or the operator of multiplication by $A_{\eps,\eta}$ which is a bounded operator. Indeed, $A_{\eps,\eta}$ belongs to $\Hunper\subset L^\infty$ since $n \in \Hunper$ according to the hypotheses and since $\varrho_{\eps,\eta} \in \calE_+$. 
 
Now, we have the tools to conclude: since $\varrho_{\eps,\eta}$ is the minimizer of \fref{minregeta} and since $\varrho_{\eps,\eta}+t\sqrt{\varrho_{\eps,\eta}}h\sqrt{\varrho_{\eps,\eta}}$ belongs to $\calE_+$ for $t$ small enough, the G\^ateaux derivative \fref{exprDF} vanishes and for all $h\in \calL(L^2)$, self-adjoint, for all $\eta\in (0,1]$, we have
 \be
 \label{char1}
\Tr \left(\sqrt{\varrho_{\eps,\eta}}\left(\beta'_\eta(\varrho_{\eps,\eta}) + H +A_{\eps,\eta} \right)\sqrt{\varrho_{\eps,\eta}}\,h\right)=0.
 \ee
 
 \bs
 \ni
 {\em Step 3: convergence of $\varrho_{\eps,\eta}$ as $\eta\to 0$.} From the estimate \fref{estimeta} and from Lemma \ref{strongconv}, one deduces that there exists $\widetilde \varrho\in \calE_+$ (dependent of $\eps$) such that, as $\eta\to 0+$,
\be
\label{coneta}
\varrho_{\eps,\eta}\to\widetilde \varrho\mbox{ in }\calJ_1\quad \mbox{and} \quad \Tr(\sqrt{H}\widetilde \varrho\sqrt{H})\leq \liminf_{\eta\to 0+} \Tr(\sqrt{H}\varrho_{\eps,\eta}\sqrt{H}).
\ee
Then, from Lemma \ref{propentropieeps} {\em (iv)}, one deduces that
\be
\label{coneta2}
\lim_{\eta\to 0} \beta_\eta (\varrho_{\eps,\eta})=\beta_0(\widetilde \varrho).  
\ee
Moreover, one can deduce from \fref{coneta} and from Sobolev embeddings in dimension one, exactly as to prove \fref{con4}, that
$$ \| n[\varrho_\eta]-n\|^2_{L^2}\to  \| n[\widetilde\varrho]-n\|^2_{L^2}$$
as $\eta\to 0$. Together with \fref{coneta} and \fref{coneta2}, this leads to
\be
\label{coneta3}
F_{\eps,0}(\widetilde \varrho)\leq \lim_{\eta\to 0+}F_{\eps,\eta}(\varrho_{\eps,\eta}).
\ee
Moreover, by definition of $\varrho_{\eps,0}$ and $\varrho_{\eps,\eta}$ as minimizers of $F_{\eps,0}$ and $F_{\eps,\eta}$, one has
$$F_{\eps,0}(\varrho_{\eps,0})\leq F_{\eps,0}(\widetilde \varrho)\quad\mbox{and}\quad F_{\eps,\eta}(\varrho_{\eps,\eta})\leq F_{\eps,\eta}(\varrho_{\eps,0}).$$
Applying Lemma \ref{propentropieeps} {\em (iv)}, one gets
$$\lim_{\eta\to 0+}F_{\eps,\eta}(\varrho_{\eps,0})=F_{\eps,0}(\varrho_{\eps,0}),$$
and finally all these limits are equal, since
$$F_{\eps,0}(\varrho_{\eps,0})\leq F_{\eps,0}(\widetilde \varrho)\leq \lim_{\eta\to 0+}F_{\eps,\eta}(\varrho_{\eps,\eta})\leq \lim_{\eta\to 0+}F_{\eps,\eta}(\varrho_{\eps,0})=F_{\eps,0}(\varrho_{\eps,0}).$$
Hence, by uniqueness of the minimizer, we have $\widetilde \varrho=\varrho_{\eps,0}$. Moreover, we deduce also from $F_{\eps,0}(\varrho_{\eps,0})= \lim_{\eta\to 0+}F_{\eps,\eta}(\varrho_{\eps,\eta})$ that
$$\Tr(\sqrt{H}\varrho_{\eps,0}\sqrt{H})= \lim_{\eta\to 0+} \Tr(\sqrt{H}\varrho_{\eps,\eta}\sqrt{H}).$$
Hence, by applying the second part of Lemma \ref{strongconv} we get finally
\be
\label{conveta}
\varrho_{\eps,\eta} \to \varrho_{\eps,0} \mbox{ in }\calJ_1\quad \mbox{and}\quad  \sqrt{H}\sqrt{\varrho_{\eps,\eta}}\to \sqrt{H}\sqrt{\varrho_{\eps,0}}\mbox{ in }\calJ_2\quad \mbox{as }\eta\to 0.
\ee

Now we have the tools to pass to the limit in \fref{char1} as $\eta \to 0+$. First, let us prove that, for all bounded operator $h$,
\be
\label{c1}
\lim_{\eta\to 0}\Tr \left(\beta'_\eta(\varrho_{\eps,\eta}) \varrho_{\eps,\eta}\,h\right)=\Tr \left(\beta'_0(\varrho_{\eps,0})\varrho_{\eps,0}\,h\right).
\ee
To this aim, we introduce a parameter $\kappa>0$ and decompose
\be
\label{rr}
\Tr \left(\beta'_\eta(\varrho_{\eps,\eta})\varrho_{\eps,\eta}\,h\right)=\Tr \left(\un_{\varrho_{\eps,\eta}\geq \kappa}\,\beta'_\eta(\varrho_{\eps,\eta})\varrho_{\eps,\eta}\,h\right)+\Tr \left(\un_{\varrho_{\eps,\eta}< \kappa}\,\beta'_\eta(\varrho_{\eps,\eta})\varrho_{\eps,\eta}\,h\right).
\ee
Since $\beta_\eta'(s)=\log(s+\eta)$, on all interval $[\kappa,M]$ with $0<\kappa<M$, one has 
$$\lim_{\eta\to 0}\max_{s\in [\kappa,M]}|\beta_\eta'(s)-\beta'_0(s)|=0\quad \mbox{and}\quad
|s\beta'_\eta(s)-s\beta'_0(s)|\leq Cs.
$$
The first term in this decomposition \fref{rr} can thus be uniformly approximated, for $M$ large enough:
$$
\begin{array}{l}
\ds \left|\Tr \left(\un_{\varrho_{\eps,\eta}\geq \kappa}\,\beta'_\eta(\varrho_{\eps,\eta})\varrho_{\eps,\eta}\,h\right)-\Tr \left(\un_{\varrho_{\eps,0}\geq \kappa}\,\beta'_0(\varrho_{\eps,0})\varrho_{\eps,0}\,h\right)\right|\\[2mm]
\ds \qquad \leq 
\|\varrho_{\eps,\eta}\|_{\calJ_1}\|h\|\max_{s\in [\kappa,M]}|\beta_\eta'(s)-\beta'_0(s)|+C_\kappa\|\varrho_{\eps,\eta}-\varrho_{\eps,0}\|_{\calJ_1}\|h\|,
\end{array}
$$
and converges to 0 as $\eta\to 0$ (for all fixed $\kappa>0$). Consider now the second term in the right-hand side of \fref{rr}. We have a uniform bound $s^{1/8}\beta'_\eta(s)\leq R$ for $s\in (0,M]$, so 
\bee
\Tr \left(\un_{\varrho_{\eps,\eta}< \kappa}\,\beta'_\eta(\varrho_{\eps,\eta})\varrho_{\eps,\eta}\,h\right)
&\leq& R\|h\|\,\|\un_{\varrho_{\eps,\eta}< \kappa}\,(\varrho_{\eps,\eta})^{7/8}\|_{\calJ_1}\\
&\leq & CR\|h\|\kappa^{1/8}\sum_{\lambda_p<\kappa}(\lambda_p)^{3/4}\\
&\leq & CR\|h\|\kappa^{1/8}\left(\sum_{p\geq 1}\frac{1}{p^{6}}\right)^{1/4},
\eee
where we used again the bound \fref{ineq2} for the eigenvalues $\lambda_p$ of $\varrho_{\eps,\eta}$, together with the estimate \fref{estimeta}. Hence
$$\lim_{\kappa\to 0}\sup_{\eta\in (0,1]}\left|\Tr \left(\un_{\varrho_{\eps,\eta}< \kappa}\,\beta'_\eta(\varrho_{\eps,\eta})\varrho_{\eps,\eta}\,h\right)\right|= 0.$$
This ends the proof of \fref{c1}.

Second, by \fref{sqrt} and by Sobolev embedding, we have
$$A_{\eps,\eta}(x)=\frac{1}{\eps} \left(n[\varrho_{\eps,\eta}]-n\right)(x)\to A_{\eps,0}(x):=\frac{1}{\eps} \left(n[\varrho_{\eps,0}]-n\right)(x)$$
in the $L^\infty(0,1)$ topology. Hence, the corresponding multiplication operators satisfy
$$A_{\eps,\eta}\to A_{\eps,0}\mbox{ in }\calL(L^2(0,1))$$
and the convergence of $\sqrt{\varrho_{\eps,\eta}}$ in $\calJ_2$ yields
\be
\label{c2}
\lim_{\eta\to 0}\Tr \left(\sqrt{\varrho_{\eps,\eta}}\,A_{\eps,\eta}\,\sqrt{\varrho_{\eps,\eta}}\,h\right)=\Tr \left(\sqrt{\varrho_{\eps,0}}\,A_{\eps,0}\,\sqrt{\varrho_{\eps,0}}\,h\right).
\ee

Third, the convergence of $\sqrt{H}\sqrt{\varrho_{\eps,\eta}}$ in $\calJ_2$ yields
\be
\label{c3}
\lim_{\eta\to 0}\Tr \left(\sqrt{\varrho_{\eps,\eta}}\sqrt{H}\,\sqrt{H}\,\sqrt{\varrho_{\eps,\eta}}\,h\right)=\Tr \left(\sqrt{\varrho_{\eps,0}}\sqrt{H}\,\sqrt{H}\,\sqrt{\varrho_{\eps,0}}\,h\right).
\ee
Finally, one can pass to the limit in \fref{char1} and \fref{c1}, \fref{c2}, \fref{c3} give, for all $h\in \calL(L^2)$ self-adjoint,
$$
\Tr \left(\sqrt{\varrho_\eps}\left(\log(\varrho_\eps) + H +A_\eps \right)\sqrt{\varrho_\eps}\,h\right)=0.
$$
where we have denoted $\varrho_\eps=\varrho_{\eps,0}$ and $A_\eps=A_{\eps,0}$. This means that
\be
 \label{char2}
\sqrt{\varrho_\eps}\left(\log(\varrho_\eps) + H +A_\eps \right)\sqrt{\varrho_\eps}=0.
 \ee

 \bs
 \ni
 {\em Step 4: the kernel of $\varrho_\eps$ is $\{0\}$.} In this step, we will prove that, for all $\eps \in (0,1]$, the kernel of the minimizer $\varrho_\eps$ of $F_\eps$ is $\{0\}$.

Let us prove this result by contradiction. Assume that the kernel of $\varrho_\eps$ is not $\{0\}$ and pick a basis function $\phi \in \mbox{Ker} \varrho_\eps$. We first complete $\phi$ into an orthonormal basis $\{\phi, (\psi_p)_{p\in I}\}$ of $\mbox{Ker}\varrho_\eps$ ($I$ may be empty, finite or infinite). Then, we denote by $(\lambda_p)_{1\leq p\leq N}$ the nonincreasing sequence of nonzero eigenvalues of $\varrho_\eps$ (here $N$ is finite or not), associated to the orthonormal family of eigenfunctions $(\phi_p)_{1\leq p\leq N}$. We thus obtain a Hilbert basis $\{\phi, (\psi_p)_{p\in I}, (\phi_p)_{1\leq p\leq  N}\}$ of $L^2(0,1)$. Since it is not clear whether $\phi$ belongs to $\Hunper$, let us regularize it by setting
$$\phi^\alpha=(1+\alpha\sqrt{H})^{-1}\phi,$$
where $\alpha>0$ is a small parameter. We have $\phi^\alpha\in \Hunper$ and, as in the proof of Lemma \ref{strongconv},
$$\lim_{\alpha\to 0}\phi^\alpha\to \phi$$
in $L^2(0,1)$. We simply fix $\alpha>0$ such that $|(\phi^\alpha,\phi)|>1/2$. Denote by $P^\alpha$ the orthogonal projection
$$
P^\alpha \varphi := \phi^\alpha\, (\phi^\alpha, \varphi), \qquad \forall \varphi \in L^2(0,1),
$$
and consider the positive operator $\varrho(t)=\varrho_\eps +t P^\alpha$ for $t>0$. From $\phi^\alpha\in \Hunper$, we deduce that the operator $P^\alpha$ belongs to $\calE_+$. We shall prove that there exists $t>0$ such that
\be
\label{co}
F_\eps(\varrho(t))<F_\eps(\varrho_\eps),
\ee
which is a contradiction.

Let $\eta>0$ and denote as before $\beta(s)=s\log s -s$ and $\beta_\eta(s)=(s+\eta) \log (s+\eta)-s-\eta\log \eta$. From the min-max principle and from the positivity of the operator $P^\alpha$, one deduces that
$$\forall p\in \NN^*,\quad \lambda_p(\varrho(t))\geq \lambda_p(\varrho_\eps),$$
where $\lambda_p(\cdot)$ denotes the $p$-th eigenvalue of the operator. Hence, we have
\bee
\beta(\lambda_p(\varrho(t))-\beta(\lambda_p(\varrho_\eps))&=&\int_{\lambda_p(\varrho_\eps)}^{\lambda_p(\varrho(t))}\log(s)ds\\
&\leq&\int_{\lambda_p(\varrho_\eps)}^{\lambda_p(\varrho(t))}\log(s+\eta)ds\\
&=&\beta_\eta(\lambda_p(\varrho(t))-\beta_\eta(\lambda_p(\varrho_\eps)),
\eee
which implies
$$\Tr(\beta(\varrho(t))-\Tr(\beta(\varrho_\eps))\leq \Tr(\beta_\eta(\varrho(t))-\Tr(\beta_\eta(\varrho_\eps))$$
and then
$$F_\eps(\varrho(t))-F_\eps(\varrho_\eps)\leq F_{\eps,\eta}(\varrho(t))-F_{\eps,\eta}(\varrho_\eps).$$
Therefore, to prove \fref{co}, it suffices to find $\eta>0$ and $t>0$  such that
\be
\label{co2}
F_{\eps,\eta}(\varrho(t))<F_{\eps,\eta}(\varrho_\eps).
\ee
Since $P^\alpha$ belongs to $\calE_+$ and by Lemma \ref{lemdiff}, for all $\eta>0$ one can differentiate $F_{\eps,\eta}$ at $\varrho_\eps$ in the direction $P^\alpha$ and one has
$$
\lim_{t\to 0}\frac{F_{\eps,\eta}(\varrho(t))-F_{\eps,\eta}(\varrho_\eps)}{t} =\Tr \left(\log( \varrho_\eps+\eta) P^\alpha \right)+ \Tr (\sqrt{H} P^\alpha \sqrt{H})+\Tr (A_\eps P^\alpha).
$$
One has
\bee
\Tr \left(\log( \varrho_\eps+\eta) P^\alpha \right)
&=&|(\phi^\alpha,\phi)|^2\log \eta+\sum_{p=1}^N|(\phi^\alpha,\phi_p)|^2\log(\lambda_p+\eta)\\
&&+\sum_{p\in I}|(\phi^\alpha,\psi_p)|^2\log\eta,
\eee
hence, by using $|(\phi^\alpha,\phi)|>1/2$, one obtains for $0<\eta<1/2$
$$
\Tr \left(\log( \varrho_\eps+\eta) P^\alpha \right)
\leq\frac{1}{4}\log \eta+\sum_{p=1}^{p_0}|(\phi^\alpha,\phi_p)|^2\log(\lambda_p+\frac{1}{2}),
$$
where $p_0$ has been chosen such that $\lambda_p<1/2$ for $p>p_0$. Therefore,  there exists a constant $C_{\eps,\alpha}$ independent of $\eta$ (but depending on $\eps$ and $\alpha$) such that
$$
\lim_{t\to 0}\frac{F_{\eps,\eta}(\varrho(t))-F_{\eps,\eta}(\varrho_\eps)}{t} \leq \frac{1}{4}\log \eta+C_{\eps,\alpha}.
$$
To conclude, it suffices to choose $\eta$ small enough such that $\frac{1}{4}\log \eta+C_{\eps,\alpha}<0$. Then one has
$$
\lim_{t\to 0}\frac{F_{\eps,\eta}(\varrho(t))-F_{\eps,\eta}(\varrho_\eps)}{t} <0,
$$
and for $t$ small enough one has \fref{co2}, which leads to a contradiction. This ends the proof of the claim. 

 \bs
 \ni
 {\em Step 5: identification of $\varrho_\eps$.} Notice that, since $A_\eps\in \Hunper\subset L^\infty(0,1)$, the operator $H_A:=H+A_\eps$ with domain $D(H)$ is bounded from below and has a compact resolvent. Denote by $(\lambda_p^\eps,\phi_p^\eps)_{p\in \NN^*}$ the eigenvalues and eigenfunctions of $\varrho_\eps$. From the previous step, we know that, for all $p$, we have $\lambda_p^\eps>0$. Moreover, $(\phi_p^\eps)_{p\in \NN^*}$ is a Hilbert basis of $L^2(0,1)$. We will prove in this step that $(\phi_p^\eps)_{p\in \NN^*}$ is a complete family of eigenfunctions of $H_A$ associated to the eigenvalues $-\log \lambda_p^\eps$.

 Apply \fref{char2} to $\phi_p^\eps$. Since from Step 4 we know that $\lambda_p^\eps>0$, we obtain
 $$\sqrt{\varrho_\eps}\left(\log(\lambda_p^\eps) + H +A_\eps \right)\phi_p^\eps=0.$$
Remark that, since $\sqrt{H}\sqrt{\varrho_\eps}$ is bounded (with adjoint operator $\sqrt{\varrho_\eps}\sqrt{H}$), we know that $\phi_p^\eps$ belongs to $\Hunper$.
Taking the $L^2$ scalar product of the above equation with $\phi_q^\eps$ leads to
\bee
0&=&\left(\sqrt{\varrho_\eps}\left(\log(\lambda_p^\eps) + H +A_\eps \right)\phi_p^\eps,\phi_q^\eps\right)\\
&=&\sqrt{\lambda_q^\eps}\log(\lambda_p^\eps)\delta_{pq}+\left(\sqrt{\varrho_\eps}\sqrt{H}\sqrt{H}\phi_p^\eps,\phi_q^\eps\right)+\left(\sqrt{\varrho_\eps}\,A_\eps\,\phi_p^\eps,\phi_q^\eps\right)\\
&=&\sqrt{\lambda_q^\eps}\log(\lambda_p^\eps)\delta_{pq}+\left(\sqrt{H}\phi_p^\eps,\sqrt{H}\sqrt{\varrho_\eps}\phi_q^\eps\right)+\left(A_\eps\,\phi_p^\eps,\sqrt{\varrho_\eps}\phi_q^\eps\right)\\
&=&\sqrt{\lambda_q^\eps}\left(\log(\lambda_p^\eps)\delta_{pq}+\left(\sqrt{H}\phi_p^\eps,\sqrt{H}\phi_q^\eps\right)+\left(A_\eps\,\phi_p^\eps,\phi_q^\eps\right)\right).
\eee
Hence, for all $p,q\in \NN^*$,
\be
\label{orth}
\left(\sqrt{H}\phi_p^\eps,\sqrt{H}\phi_q^\eps\right)+\left(A_\eps\,\phi_p^\eps,\phi_q^\eps\right)=-\log(\lambda_p^\eps)\delta_{pq}.
\ee
The family $(\phi_p^\eps)_{p\in \NN^*}$ is thus an orthogonal family for the following sesquilinear form associated to $H_A$:
$$Q_A(u,v)=\left(\sqrt{H}u,\sqrt{H}v\right)+\left(A_\eps\,u,v\right).$$
Note that, since $A_\eps\in L^\infty$, there exists two constants $m,M>0$ such that
\be
\label{QA}
\forall u\in \Hunper,\qquad \frac{1}{M}\|u\|_{\H^1}^2\leq Q_A(u,u)+m\|u\|_{L^2}^2\leq M\|u\|_{\H^1}^2.
\ee

Let us now prove that this family $(\phi_p^\eps)_{p\in \NN^*}$ is dense in $\Hunper$. Let $\phi\in \Hunper$. We already know that the following series:
$$\phi_N=\sum_{p=1}^N (\phi_p^\eps,\phi)\phi_p^\eps$$
converges in $L^2(0,1)$ to $\phi$ as $N\to+\infty$. We will prove that in fact this series converges in $\H^1$ which, by \fref{QA}, is equivalent to saying that
\be
\label{lim}
Q_A(\phi,\phi)=\lim_{N\to +\infty}Q_A(\phi_N,\phi_N).
\ee
Again, the key argument of the proof will be the fact that $\varrho_\eps$ is the minimizer of $F_\eps$: for all $t>0$, we have
\be
\label{cont}
0\leq F_\eps(\varrho_\eps +t P)-F_\eps(\varrho_\eps)
\ee
where  $P$ denotes the orthogonal projection on $\phi$:
$$
P u := \phi\, (\phi, u), \qquad \forall u \in L^2(0,1).
$$
Indeed, $\phi\in \Hunper$ implies that $P\in \calE_+$, thus $\widetilde \varrho(t):=\varrho_\eps+tP$ belongs to $\calE_+$ for all $t>0$. Now, as in the previous Step 4, one can prove that, for all $\eta>0$, we have
$$
0\leq F_\eps(\varrho_\eps +t P)-F_\eps(\varrho_\eps)\leq F_{\eps,\eta}(\varrho_\eps +t P)-F_{\eps,\eta}(\varrho_\eps)
$$
and
\bee
\lim_{t\to 0}\frac{F_{\eps,\eta}(\varrho_\eps +t P)-F_{\eps,\eta}(\varrho_\eps)}{t} 
&=&\Tr \left(\log(\varrho_\eps+\eta) P \right)+ \Tr (\sqrt{H} P \sqrt{H})+\Tr (A_\eps P)\\
&=&\sum_{p\in \NN^*}|(\phi,\phi_p^\eps)|^2\log(\lambda_p^\eps+\eta)+Q_A(\phi,\phi).
\eee
Therefore, for all $\eta>0$, one has
\be
\label{co3}
-\sum_{p\in \NN^*}|(\phi,\phi_p^\eps)|^2\log(\lambda_p^\eps+\eta)\leq Q_A(\phi,\phi).
\ee
Let $N\in \NN^*$ large enough, such that $\log( \lambda_N^\eps)<0$. For $\eta>0$ small enough, one has
$$\forall p\geq N,\qquad \log(\lambda_p^\eps+\eta)\leq 0,$$
thus \fref{co3} yields
$$-\sum_{p=1}^N|(\phi,\phi_p^\eps)|^2\log(\lambda_p^\eps+\eta)\leq Q_A(\phi,\phi).$$
Since we know that $\lambda_p^\eps>0$ for all $p$, one can pass to the limit in this inequality as $\eta\to 0$:
$$-\sum_{p=1}^N|(\phi,\phi_p^\eps)|^2\log\lambda_p^\eps\leq Q_A(\phi,\phi).$$
Remarking that, by the orthogonality property \fref{orth}, one has
$$Q_A(\phi_N,\phi_N)=-\sum_{p=1}^N |(\phi,\phi_p^\eps)|^2\log \lambda_p^\eps,$$
this inequality reads
\be
\label{leq}
Q_A(\phi_N,\phi_N)\leq Q_A(\phi,\phi).
\ee
In particular, this means that $(\phi_N)$ is a bounded sequence of $\Hunper$, thus converges weakly to $\phi$ in $\H^1$ as $N\to +\infty$. From the equivalence of norms \fref{QA}, we then deduce that
$$Q_A(\phi,\phi)\leq \liminf_{N\to +\infty}Q_A(\phi_N,\phi_N).$$
Together with \fref{leq}, we get \fref{lim} and our claim is proved: $\phi_N$ converges to $\phi$ in the $\H^1$ strong topology and the family $(\phi_p^\eps)_{p\in \NN^*}$ is dense in $\Hunper$.

This enables to conclude the proof. Indeed, this density property implies that \fref{orth} is equivalent to
\be
\label{QAphi}
\forall \phi\in \Hunper,\qquad \left(\sqrt{H}\phi_p^\eps,\sqrt{H}\phi\right)+\left(A_\eps\,\phi_p^\eps,\phi\right)=-(\phi_p^\eps,\phi) \log \lambda_p^\eps.
\ee
This means that $(\phi_p^\eps)_{p\in \NN^*}$ is a complete family of eigenfunctions of $H_A$ (still identified with the associated quadratic form) and that the associated eigenvalues of $H_A$ are $-\log \lambda_p^\eps$. In other words, we have
$$\varrho_\eps =\exp \left(-H_A\right)$$
in the sense of functional calculus. The proof of Proposition \ref{lpropminreg} is complete.
\end{proof}

\subsection{Passing to the limit}
\label{subse}
In this subsection, we terminate the proof of our main Theorem \ref{theo1}. Let $n(x)$ such that $n>0$ on $[0,1]$ and $\sqrt{n}\in\Hunper$. By a Sobolev embedding in dimension one, $\sqrt{n}$ is continuous on $[0,1]$, so that we have
\be
\label{minn}
0<m:=\min_{x\in[0,1]}n(x).
\ee
Propositions \ref{propmin2} and \ref{lpropminreg} define the unique minimizers $\varrho[n]$ and $\varrho_\eps[n]$ -- shortly denoted $\varrho$ and $\varrho_\eps$ here -- of the minimization problems \fref{min2} and \fref{minreg}. Moreover, $\varrho_\eps$ takes the form \fref{charreg}, with $$A_\eps=\frac{1}{\eps}(n^\eps-n)\in \Hunper,$$
where we have denoted $n_\eps:=n[\varrho_\eps]$. Let us study successively the limits of $\varrho_\eps$, $n_\eps$ and $A_\eps$ as $\eps \to 0$.

\bs
\ni
{\em Step 1: convergence of $\varrho_\eps$.} Recall that, for all $\sigma\in \calE_+$, we have
$$F_\eps(\sigma)=F(\sigma)+\frac{1}{2\eps}\|n[\sigma]-n\|_{L^2}^2$$
and that $n[\varrho]=n$. Hence, by definition of $\varrho_\eps$, we have
\bea
F(\varrho_\eps)\leq F_\eps(\varrho_\eps)&=&\Tr(\varrho_\eps\log \varrho_\eps-\varrho_\eps)+\Tr \sqrt{H}\varrho_\eps\sqrt{H}+\frac{1}{2\eps}\|n[\varrho_\eps]-n\|_{L^2}^2\nonumber \\
&\leq& F_\eps(\varrho)=F(\varrho).\label{compar}
\eea
Therefore, one deduces from the estimate \fref{souslin} that $\Tr \sqrt{H}\varrho_\eps\sqrt{H}$ is bounded independently of $\eps$ and that $n[\varrho_\eps]$ converges to $n$ in $L^2(0,1)$. In particular, by Cauchy-Schwarz, we obtain
$$
\left|\Tr \varrho_\eps-\Tr\varrho\right|=\left|\int_0^1 (n[\varrho_\eps]-n)(x)dx\right|\leq  \|n[\varrho_\eps]-n\|_{L^2}\to 0\quad \mbox{as }\eps \to 0.
$$
The family $\varrho_\eps$ is thus bounded in $\calE_+$ independently of $\eps$ and then, by Lemma \ref{strongconv}, there exists $\widetilde \varrho\in \calE_+$ such that
\be
\label{conetasuite}
\varrho_\eps\to\widetilde \varrho\mbox{ in }\calJ_1\quad \mbox{and} \quad \Tr(\sqrt{H}\widetilde \varrho\sqrt{H})\leq \liminf_{\eps\to 0} \Tr(\sqrt{H}\varrho_\eps\sqrt{H}).
\ee
Therefore, by Lemma \ref{propentropie} {\em (ii)}, by the expression \fref{F} of $F$ and by \fref{compar}, one gets
\be
\label{compar2}
F(\widetilde \varrho)\leq \liminf_{\eps\to 0}F(\varrho_\eps)\leq F(\varrho).
\ee
Furthermore, we have $n[\widetilde \varrho]=n$. Indeed, the strong $\calJ_1$ convergence of $\varrho_\eps$ implies the weak $\calJ_1$ convergence, thus
$$\forall \varphi\in L^\infty(0,1),\quad \int_0^1 n[\varrho_\eps](x)\varphi(x)dx=\Tr (\varrho_\eps\,\varphi)\to \Tr (\widetilde \varrho\,\varphi)=\int_0^1 n[\widetilde \varrho](x)\varphi(x)dx$$
and we already know that $n[\varrho_\eps]$ converges to $n$ in $L^2(0,1)$, so 
$$\forall \varphi\in L^\infty(0,1),\quad \int_0^1 n[\widetilde \varrho](x)\varphi(x)dx=\int_0^1 n(x)\varphi(x)dx$$
and then $n[\widetilde \varrho]=n$.

Finally, $\widetilde \varrho$ is a minimizer of \fref{min2} and the uniqueness of this minimizer proved in Proposition \ref{propmin2} yields $\widetilde \varrho=\varrho$. Moreover, one has $F(\varrho_\eps)\to F(\varrho)$, so
\bee
\lim_{\eps\to 0}\Tr(\sqrt{H}\varrho_\eps\sqrt{H})&=&\lim_{\eps \to 0}F(\varrho_\eps)-\lim_{\eps \to 0}\Tr (\varrho_\eps\log\varrho_\eps-\varrho_\eps)\\
&=& F(\varrho)-\Tr (\varrho\log\varrho-\varrho)=\Tr(\sqrt{H}\varrho\sqrt{H}).
\eee
The second part of Lemma \ref{strongconv} can thus be applied and one has finally
\be
\label{convetasuite}
\varrho_\eps \to \varrho \mbox{ in }\calJ_1\quad \mbox{and}\quad  \sqrt{H}\sqrt{\varrho_\eps}\to \sqrt{H}\sqrt{\varrho}\mbox{ in }\calJ_2\quad \mbox{as }\eps\to 0.
\ee
Notice that, from Lemma \ref{strongconv}, we also have
\be
\label{convetasuite2}
\sqrt{\varrho_\eps}\to \sqrt{\varrho}\mbox{ in }\calJ_2\quad \mbox{as }\eps\to 0.
\ee

\bs
\ni
{\em Step 2: convergence of $n_\eps:=n[\varrho_\eps]$.} In this step, we will prove that
\be
\label{convnH1}
\lim_{\eps\to 0}\|n_\eps-n\|_{\H^1}=0.
\ee
{}From the previous step, we know that $\varrho_\eps\to \varrho$ in $\calJ_1$. According to Lemma \ref{convFP}, for the sequel of the proof, let us choose some eigenbasis $(\phi_p^\eps)_{p\in\NN^*}$ and $(\phi_p)_{p\in\NN^*}$, respectively of $\varrho_\eps$ and $\varrho$ and such that
\be
\label{convFP2}
\forall p\in \NN^*,\quad \lim_{\eps\to 0}\|\phi_p^\eps-\phi_p\|_{L^2}=0.
\ee
We claim that
\bea
\label{cl1}
\forall p\in \NN^*\quad &&\lim_{\eps\to 0}\left\|\sqrt{\varrho_\eps}\phi_p^\eps-\sqrt{\varrho}\phi_p \right\|_{L^2}=0,\\
\label{cl2}
\mbox{and}&&\lim_{\eps\to 0}\left\|\sqrt{H}\sqrt{\varrho_\eps}\phi_p^\eps-\sqrt{H}\sqrt{\varrho}\phi_p \right\|_{L^2}=0.
\eea
Indeed, to prove \fref{cl1}, it suffices to write 
\bee
\left\|\sqrt{\varrho_\eps}\phi_p^\eps-\sqrt{\varrho}\phi_p \right\|_{L^2}
&\leq&\left\|\sqrt{\varrho_\eps}(\phi_p^\eps-\phi_p) \right\|_{L^2}+\left\|(\sqrt{\varrho_\eps}-\sqrt{\varrho})\phi_p\right\|_{L^2}\\
&\leq&\|\sqrt{\varrho_\eps}\|_{\calL(L^2)}\left\|\phi_p^\eps-\phi_p \right\|_{L^2}+\left\|\sqrt{\varrho_\eps}-\sqrt{\varrho}\right\|_{\calL(L^2)}\\
&\leq&\|\sqrt{\varrho_\eps}\|_{\calJ_2}\left\|\phi_p^\eps-\phi_p \right\|_{L^2}+\left\|\sqrt{\varrho_\eps}-\sqrt{\varrho}\right\|_{\calJ_2}
\eee
then to use \fref{convetasuite2} and \fref{convFP2}. To prove \fref{cl2}, one writes similarly
$$
\left\|\sqrt{H}\sqrt{\varrho_\eps}\phi_p^\eps-\sqrt{H}\sqrt{\varrho}\phi_p \right\|_{L^2}
\leq\|\sqrt{H}\sqrt{\varrho_\eps}\|_{\calJ_2}\left\|\phi_p^\eps-\phi_p \right\|_{L^2}+\left\|\sqrt{H}\sqrt{\varrho_\eps}-\sqrt{H}\sqrt{\varrho}\right\|_{\calJ_2}
$$
then use \fref{convetasuite} and \fref{convFP2}.

Let us prove \fref{convnH1}. We already have $n_\eps\to n$ in $L^2(0,1)$, so it remains to prove the convergence of $\frac{dn_\eps}{dx}$ in $L^2$. One has
$$n_\eps=\sum_{p\in \NN^*}\lambda_p^\eps|\phi_p^\eps|^2,\qquad n=\sum_{p\in \NN^*}\lambda_p|\phi_p|^2,$$
thus
\bee
\left|\frac{dn_\eps}{dx}-\frac{dn}{dx}\right|&\leq &2 \sum_{p \in \NN^*}\left|\sqrt{\lambda_p^\eps} \frac{d}{dx} \phi_p^\eps -\sqrt{\lambda_p} \frac{d}{dx} \phi_p\right| \left| \sqrt{\lambda_p^\eps} \phi_p^\eps\right|\\
&&+2 \sum_{p \in \NN^*}\left|\sqrt{\lambda_p} \frac{d}{dx}\phi_p\right| \left| \sqrt{\lambda_p^\eps} \phi_p^\eps -\sqrt{\lambda_p} \phi_p\right|.\end{eqnarray*}
Therefore, from \fref{ident} and Cauchy-Schwarz, one gets
$$\left\|\frac{dn_\eps}{dx}-\frac{dn}{dx}\right\|_{L^2}\leq s_1^\eps+s_2^\eps$$
with
\bea
s_1^\eps&=&\sum_{p\in\NN^*}\left\|(\sqrt{H}\sqrt{\varrho_\eps})\phi_p^\eps-(\sqrt{H}\sqrt{\varrho})\phi_p \right\|_{L^2}\left\|\sqrt{\varrho_\eps}\phi_p^\eps\right\|_{L^\infty}\label{s1}\\
s_2^\eps&=&\sum_{p\in\NN^*}\left\|(\sqrt{H}\sqrt{\varrho})\phi_p\right\|_{L^2}\left\|\sqrt{\varrho_\eps}\phi_p^\eps-\sqrt{\varrho}\phi_p\right\|_{L^\infty}.\label{s2}
\eea
Let us use again \fref{ineq2}. From the Gagliardo-Nirenberg inequality, one deduces that, for all $N\in \NN^*$,
\bee
\sum_{p\geq N}\lambda_p^\eps \| \phi_p^\eps \|^2_{L^\infty}&\leq &\sum_{p\geq N}\sqrt{\lambda_p^\eps} \| \phi_p^\eps \|_{L^2}\| \sqrt{\lambda_p^\eps} \sqrt{H}\phi_p^\eps \|_{L^2}\\
&\leq &\left(\sum_{p\geq N}\lambda_p^\eps \right)^{1/2}\left(\Tr(\sqrt{H}\varrho_\eps\sqrt{H}\right)^{1/2}\\
&\leq &C\|\varrho_\eps\|^{1/4}_{\calL(L^2)}\left(\sum_{p\geq N}p^2\lambda_p^\eps\right)^{1/4}\left(\sum_{p\geq N}\frac{1}{p^2}\right)^{1/4}\left(\Tr(\sqrt{H}\varrho_\eps\sqrt{H}\right)^{1/2}\\
&\leq& \frac{C}{N^{1/4}}(\Tr \varrho_\eps)^{1/4}\left(\Tr \sqrt{H}\varrho_\eps \sqrt{H}\right)^{3/4}\leq \frac{C}{N^{1/4}},
\eee
and similarly
$$\sum_{p\geq N}\lambda_p \| \phi_p \|^2_{L^\infty} \leq \frac{C}{N^{1/4}}.
$$
Thus, for all $N\in \NN^*$
\begin{eqnarray*}
| s_1^\eps|^2&\leq& 2 \left(\sum_{p \in \NN^*} \lambda_p^\eps \| \phi_p^\eps \|^2_{L^\infty}\right)\left(\sum_{p<N} 
\|\sqrt{H}\sqrt{\rho_\eps}  \phi_p^\eps -\sqrt{H}\sqrt{\rho}\phi_p\|^2_{L^2}\right) \\
&&+\frac{C}{N^{1/4}}\sum_{p\geq N} 
\left(\|\sqrt{H}\sqrt{\rho_\eps}  \phi_p^\eps\|_{L^2}^2+\|\sqrt{H}\sqrt{\rho}\phi_p\|^2_{L^2}\right)\\
&\leq&  C\sum_{p<N} 
\|\sqrt{H}\sqrt{\rho_\eps}  \phi_p^\eps -\sqrt{H}\sqrt{\rho}\phi_p\|^2_{L^2}+\frac{C}{N^{1/4}}
\end{eqnarray*}
and from \fref{cl2} one deduces that
\be
\label{convs1}
\lim_{\eps\to 0}| s_1^\eps|^2=0.
\ee
Furthermore,
\begin{eqnarray*}
| s_2^\eps|^2&\leq& \left(\sum_{p \in \NN^*}\|\sqrt{H}\sqrt{\varrho}\phi_p \|^2_{L^2}\right)\left(\sum_{p<N} 
\|\sqrt{\rho_\eps}  \phi_p^\eps -\sqrt{\rho}\phi_p\|^2_{L^\infty}\right)  \\
&&+\left(\sum_{p\in \NN^*}\|\sqrt{H}\sqrt{\varrho}\phi_p\|^2_{L^2}\right)\left(\sum_{p\geq N} 
\left(\|\sqrt{\rho_\eps}  \phi_p^\eps\|_{L^\infty}^2+\|\sqrt{\rho}\phi_p\|^2_{L^\infty}\right)\right)\\
&\leq& C\left(\sum_{p \in \NN^*}\|\sqrt{H}\sqrt{\varrho}\phi_p \|^2_{L^2}\right)\left(\sum_{p<N} \|\sqrt{\rho_\eps}  \phi_p^\eps -\sqrt{\rho}\phi_p\|^2_{L^2}\right)^{1/2}\times\\
&&\times\left(\sum_{p<N} \|\sqrt{H}\sqrt{\rho_\eps}  \phi_p^\eps -\sqrt{H}\sqrt{\rho}\phi_p\|^2_{L^2}\right)^{1/2}+\frac{C}{N^{1/4}}\sum_{p\in \NN^*}\|\sqrt{H}\sqrt{\varrho}\phi_p\|^2_{L^2}
\end{eqnarray*}
where we used again the Gagliardo-Nirenberg inequality. Hence, from \fref{cl2}, one deduces that
\be
\label{convs2}
\lim_{\eps\to 0}| s_2^\eps|^2=0.
\ee
The convergence \fref{convnH1} of $n_\eps$ is proved.

\bs
\ni
{\em Step 3: convergence of $A_\eps$.} By a Sobolev embedding in dimension one, $\Hunper$ is a Banach algebra: for all $u,v \in \Hunper$ the product $uv$ also belongs to $\Hunper$. Hence, from \fref{QAphi}, one gets
$$
\forall p\in \NN^*,\,\,\forall \phi\in \Hunper,\quad \left(\sqrt{H}\phi_p^\eps,\sqrt{H}(\phi_p^\eps\phi)\right)+\int_0^1A_\eps\,|\phi_p^\eps|^2\phi dx=-(\phi_p^\eps,\phi_p^\eps\phi) \log \lambda_p^\eps.
$$
Multiply this identity by $\lambda_p^\eps$ and sum up on $p$. Since $n_\eps=\sum_p\lambda_p^\eps|\phi_p^\eps|^2$, we obtain that, for all $\phi\in \Hunper$,
\bea
\int_0^1A_\eps n_\eps\phi dx
&=&-\sum_{p\in\NN^*}\left(\phi_p^\eps,\phi (\varrho_\eps\log \varrho_\eps) \phi_p^\eps\right)-\sum_{p\in\NN^*}\left(\sqrt{H}\sqrt{\varrho_\eps}\phi_p^\eps,\sqrt{H}(\phi\sqrt{\varrho_\eps}\phi_p^\eps)\right)\nonumber\\
&=&-\Tr(\phi (\varrho_\eps\log \varrho_\eps))-\sum_{p\in\NN^*}\left(\sqrt{H}\sqrt{\varrho_\eps}\phi_p^\eps,\sqrt{H}(\phi\sqrt{\varrho_\eps}\phi_p^\eps)\right).
\label{seul}
\eea
Let us examinate separately the convergence of the two terms in the right hand-side. From \fref{convetasuite}, Lemma \ref{propentropie} {\em (ii)}, and from
$$\left|\Tr(\phi (\varrho_\eps\log \varrho_\eps))-\Tr(\phi (\varrho\log \varrho))\right|\leq \|\phi\|_{L^\infty}\|\varrho_\eps\log \varrho_\eps-\varrho\log \varrho\|_{\calJ_1}$$
one has
\be
\label{weak1}
\lim_{\eps\to 0}\sup_{\|\phi\|_{\H^1}\leq 1}\left|\Tr(\phi (\varrho_\eps\log \varrho_\eps))-\Tr(\phi (\varrho\log \varrho))\right|=0.
\ee
Let us now prove that 
\be
\label{convetasuite3}
\sqrt{H}\phi\sqrt{\varrho_\eps}\to \sqrt{H}\phi\sqrt{\varrho} \mbox{ in }\calJ_2\quad\mbox{ as }\eps\to 0,
\ee
where $\phi$ denotes the operator of multiplication by $\phi$. Using the identification \fref{ident}, one gets
\bee
&&\|\sqrt{H}\phi(\sqrt{\varrho_\eps}-\sqrt{\varrho})\|_{\calJ_2}^2=\sum_p \|\sqrt{H}\phi(\sqrt{\varrho_\eps}-\sqrt{\varrho})\phi_p\|_{L^2}^2\\
&&\qquad =\sum_p \left\|\frac{d}{dx}(\phi(\sqrt{\varrho_\eps}-\sqrt{\varrho})\phi_p)\right\|_{L^2}^2\\
&&\qquad\leq 2\left\|\frac{d\phi}{dx}\right\|_{L^2}^2\sum_p \|(\sqrt{\varrho_\eps}-\sqrt{\varrho})\phi_p\|_{L^\infty}^2+2\|\phi\|_{L^\infty}^2\sum_p \left\| \frac{d}{dx}((\sqrt{\varrho_\eps}-\sqrt{\varrho})\phi_p)\right\|_{L^2}^2\\
&&\qquad\leq C\left\|\frac{d\phi}{dx}\right\|_{L^2}^2\sum_p \|(\sqrt{\varrho_\eps}-\sqrt{\varrho})\phi_p\|_{L^2} \left\|\sqrt{H}(\sqrt{\varrho_\eps}-\sqrt{\varrho})\phi_p\right\|_{L^2}\\
&&\qquad \quad +C\|\phi\|_{\H^1}^2\sum_p \left\| \sqrt{H}((\sqrt{\varrho_\eps}-\sqrt{\varrho})\phi_p)\right\|_{L^2}^2\\
&&\quad\leq C\|\phi\|_{\H^1}^2\left(\|\sqrt{\varrho_\eps}-\sqrt{\varrho}\|_{\calJ_2}^2+\|\sqrt{H}\sqrt{\varrho_\eps}-\sqrt{H}\sqrt{\varrho}\|_{\calJ_2}^2\right)
\eee
where we used a Gagliardo-Nirenberg inequality. Hence, from \fref{convetasuite} and \fref{convetasuite2}, one deduces \fref{convetasuite3}. Finally, from \fref{convetasuite} and \fref{convetasuite3}, one gets the following convergence, as $\eps\to 0$:
\bea
\sum_{p\in\NN^*}\left(\sqrt{H}\sqrt{\varrho_\eps}\phi_p^\eps,\sqrt{H}(\phi\sqrt{\varrho_\eps}\phi_p^\eps)\right)
&=&\left(\sqrt{H}\sqrt{\varrho_\eps},\sqrt{H}\phi\sqrt{\varrho_\eps}\right)_{\calJ_2}\nonumber\\
&\to&\left(\sqrt{H}\sqrt{\varrho},\sqrt{H}\phi\sqrt{\varrho}\right)_{\calJ_2}\label{weak2}\\
&&=\sum_{p\in\NN^*}\left(\sqrt{H}\sqrt{\varrho}\phi_p,\sqrt{H}(\phi\sqrt{\varrho}\phi_p)\right).\nonumber
\eea

Let us now define a linear form on $\Hunper$. For $\psi\in \Hunper$, we set
\be
\left(A,\psi\right)_{\H^{-1}_{per},\Hunper}=\Tr\left(\frac{\psi}{n} (\varrho\log \varrho)\right)+\sum_{p\in\NN^*}\left((\sqrt{H}\sqrt{\varrho})\phi_p,\sqrt{H}\left(\frac{\psi}{n}\sqrt{\varrho}\phi_p\right)\right).\label{defAA}
\ee
{}From the above estimates, one deduces that
\be
\left|\left(A,\psi\right)_{\H^{-1}_{per},\Hunper}\right|\leq C\left\|\frac{\psi}{n}\right\|_{\H^1}\left(|\Tr \varrho\log\varrho|+\Tr \varrho+\Tr \sqrt{H}\varrho\sqrt{H}\right).
\label{borneHmoinsun}
\ee
Since $n(x)\geq m>0$ on $[0,1]$, the application $\psi\mapsto \frac{\psi}{n}$ is continuous on $\Hunper$, so the above defined linear form $A$ is continuous on $\Hunper$ and belongs to its dual space $\H^{-1}_{per}$. Moreover, we have proved by \fref{seul}, \fref{weak1} and \fref{weak2} that, for all $\phi\in\Hunper$, 
$$\lim_{\eps\to 0}\sup_{\|\phi\|_{\H^1}\leq 1}\left|\int_0^1 A_\eps n_\eps \phi dx-\left(A,n\phi\right)_{\H^{-1}_{per},\Hunper}\right|=0.$$
To conclude, it remains to use the convergence \fref{convnH1} of $n_\eps$ to $n$ in $\H^1$, which implies that $\frac{1}{n_\eps}$ converges to $\frac{1}{n}$ and that, in fact,
$$\lim_{\eps\to 0}\sup_{\|\psi\|_{\H^1}\leq 1}\left|\int_0^1 A_\eps \psi dx-\left(A,\psi\right)_{\H^{-1}_{per},\Hunper}\right|=0.$$
In other words, one has
\be
\label{convHmoinsun}
A_\eps \to A \quad \mbox{in the $\H^{-1}_{per}$ strong topology as $\eps\to 0$}.
\ee

\bs
\ni
{\em Step 4: identification of $\varrho$ and conclusion.} Let us define the following forms on $\Hunper$,
\bee
Q_{A_\eps}(\varphi,\psi)&=&(\sqrt{H} \varphi,\sqrt{H} \psi)+(A_\eps \varphi,\psi)\\
Q_{A}(\varphi,\psi)&=&(\sqrt{H} \varphi,\sqrt{H} \psi)+\left( A,  \overline{\varphi} \psi \right)_{\H^{-1}_{per},\Hunper}.
\eee
The form $\left( A,  \overline{\varphi} \psi \right)_{\H^{-1}_{per},\Hunper}$ is a symmetric, form-bounded perturbation of $(\sqrt{H} \varphi,\sqrt{H} \psi)$ with relative bound $< 1$. Indeed, by \fref{borneHmoinsun} and by $\frac{1}{n}\in \Hunper$,
\begin{eqnarray*}
\left( A,  |\varphi|^2 \right)_{\H^{-1}_{per},\Hunper} &\leq& C\| |\varphi|^2 \|_{\H^1}\leq C\| \varphi \|_{L^4}^2+C\|\varphi\|_{L^\infty}\left\|\frac{d\varphi}{dx}\right\|_{L^2}\\
&\leq & C \| \varphi \|_{L^2}\left\|\frac{d\varphi}{dx}\right\|_{L^2}+C\| \varphi \|_{L^2}^{1/2}\left\|\frac{d\varphi}{dx}\right\|^{3/2}_{L^2}\\
&\leq& \frac{1}{2}(\sqrt{H} \varphi,\sqrt{H} \varphi)+C\|\varphi \|^{2}_{L^{2}},
\end{eqnarray*}
where used the Young inequality and a standard Gagliardo-Nirenberg inequality. Let $H_A$ be the unique self-adjoint operator associated to $Q_A$. Then, according to Theorem XIII.68 of \cite{RS-80-4} $H_A$ has a compact resolvent and we denote by $(\mu_p)_{p \in \NN^*}$ its eigenvalues. In addition, we have
\begin{eqnarray*}
|Q_{A_\eps}(\varphi,\varphi)-Q_A(\varphi,\varphi)|&\leq& C\|A_\eps-A\|_{\H^{-1}_{per}}\| |\varphi|^2 \|_{\H^1},\\
& \leq& C \|A_\eps-A\|_{\H^{-1}_{per}} \left((\sqrt{H} \varphi, \sqrt{H} \varphi)+\|\varphi\|_{L^2}^2\right).
\end{eqnarray*}
Moreover, Theorem 3.6 of \cite{kato}, chapter VI, section 3.2 yields the convergence of operators in the generalized sense, which implies in particular the convergence of the eigenvalues:
$$ 
-\log \lambda_p^\eps=\mu_p^\eps \to \mu_p \qquad \forall p \in \NN^*
$$
as $\eps\to 0$.
Hence, by continuity of the exponential function,
$$\lambda_p^\eps=\exp(-\mu_p^\eps)\to \exp(-\mu_p) \qquad \forall p \in \NN^*.$$
Besides, according to Lemma \ref{convFP} in the Appendix, the $\calJ_1$ convergence of $\varrho_\eps$ to $\varrho$ implies the convergence of the eigenvalues:
$$\lambda_p^\eps\to \lambda_p \qquad \forall p \in \NN^*.$$
This enables to completely identify the eigenvalues of $\varrho$: we have
$$\lambda_p=\exp(-\mu_p).$$
Furthermore, from \fref{cl2} and from $\lambda_p>0$, one deduces that $\phi_p^\eps\to \phi_p$ in $\H^1$. One can thus pass to the limit in \fref{QAphi}: for all $p\in \NN^*$ and for all $\phi\in \Hunper$,
$$\mu_p^\eps (\phi_p^\eps,\phi)=Q_{A_\eps}(\phi_p^\eps,\phi)\to Q_A(\phi_p,\phi),$$ 
which yields
$$Q_A(\phi_p,\phi)=\mu_p (\phi_p,\phi),\quad \forall \phi\in \Hunper,$$ 
so, finally, $(\phi_p)_{p\in \NN^*}$ is the complete basis of eigenvalues of $Q_A$. We have completely identified $\varrho$:
$$\varrho=\exp\left(-(H+A)\right).$$
The proof of our main Theorem \ref{theo1} is complete.

\begin{appendix} 
\section{}
\begin{lemma} \label{lieb} Let $\varrho \in \calE_+$ and denote by $(\rho_p)_{p\geq 1}$ the nonincreasing sequence of nonzero eigenvalues of $\varrho$, associated to the orthonormal family of eigenfunctions $(\phi_p)_{p\geq 1}$. Denote by $(\lambda_p[H])_{p \geq 1}$ the nondecreasing sequence of eigenvalues of the Hamiltonian $H$. Then we have
$$
\Tr (\sqrt{H} \rho \sqrt{H})=\sum_{p \geq 1} \rho_p \,(  \sqrt{H} \phi_p, \sqrt{H}\phi_p) \geq \sum_{p \geq 1} \rho_p \,\lambda_p[H].
$$
\end{lemma} 
\begin{proof}
Notice first that
\begin{eqnarray*}
\Tr (\sqrt{H} \rho \sqrt{H})&=&\Tr (\sqrt{H} \sqrt{\rho} (\sqrt{H} \sqrt{\rho})^*)=\sum_{p \geq 1}(  \sqrt{H} \sqrt{\rho}\phi_p,  \sqrt{H} \sqrt{\rho} \phi_p),\\
&=&\sum_{p \geq 1} \rho_p \,(  \sqrt{H} \phi_p, \sqrt{H}\phi_p).
\end{eqnarray*}
Then,
\begin{eqnarray*}
\sum_{p=1}^N \rho_p \,( \sqrt{H} \phi_p, \sqrt{H} \phi_p) &=&
 \rho_N \sum_{p =1}^N(  \sqrt{H} \phi_p, \sqrt{H} \phi_p)+ (\rho_{N-1}-\rho_{N})\sum_{i =1}^{N-1} (  \sqrt{H} \phi_p, \sqrt{H}\phi_p)\\
&&+\cdots (\rho_{2}-\rho_{3})\sum_{p =1}^{2}(  \sqrt{H} \phi_p, \sqrt{H} \phi_p)+\rho_1(  \sqrt{H} \phi_1,  \sqrt{H}\phi_1).
\end{eqnarray*}
Using \cite{Lieb-Loss}, Theorem 12.1 page 300, it comes
$$
\sum_{p =1}^N (  \sqrt{H} \phi_p,  \sqrt{H} \phi_p) \geq \sum_{p =1}^N \lambda_p[H],
$$
so that, since $\rho_p \leq \rho_{p-1}$, $\forall p \geq 1$,
\begin{eqnarray*}
\sum_{p =1}^N \rho_p \,(  \sqrt{H} \phi_p, \sqrt{H}\phi_p) &\geq &
 \rho_N \sum_{p =1}^N  \lambda_p[H]+ (\rho_{N-1}-\rho_{N})\sum_{p =1}^{N-1} \lambda_p[H]\\
&&+\cdots (\rho_{2}-\rho_{3})\sum_{p =1}^{2} \lambda_p[H]+\rho_1  \lambda_1[H],\\
&=&\sum_{p=1}^N \rho_p \,\lambda_p[H].
\end{eqnarray*}
We conclude by passing to the limit as $N\to +\infty$. Notice that the theorem of  \cite{Lieb-Loss} is written for Hamiltonians defined on $\RR^d$, but it can be easily extended to bounded domains. This ends the proof of the lemma.
\end{proof}

\bs
\begin{lemma} 
\label{convFP}
Let a sequence $\varrho_k$ converging to $\varrho$ in $\calJ_1$ as $k\to +\infty$. Then the corresponding nonincreasing sequence of eigenvalues $(\lambda_p^k)_{p\in \NN^*}$, $(\lambda_p)_{p\in \NN^*}$ converge as follows:
$$\forall p\in \NN^*,\quad \lim_{k\to +\infty}\lambda_p^k\to\lambda_p.$$
Moreover, there exist a sequence of orthonormal eigenbasis $(\phi_p^k)_{p\in \NN^*}$ of $\varrho_k$ and an orthonormal eigenbasis $(\phi_p)_{p\in \NN^*}$ of $\varrho$ such that
$$\forall p\in \NN^*,\quad \lim_{k\to +\infty}\|\phi_p^k-\phi_p\|_{L^2}=0.$$
\end{lemma}
\begin{proof}
Let us first prove the convergence of the eigenvalues. According to \cite{Simon-trace}, Theorem 1.20, we have the following relation between the eigenvalues of $\varrho_k$ and $\varrho$:
$$
\lambda_p^k-\lambda_p=\sum_{q=1}^{\infty} \alpha_{pq} \lambda_q[\varrho_k-\varrho], \qquad p \geq 1,
$$
where $(\lambda_q[\varrho_k-\varrho])_{q \geq 1}$ denote the eigenvalues of $\varrho_k-\varrho$ and $\alpha$ is a doubly stochastic matrix, that is a matrix with positive entries such that 
$\sum_{p=1}^{\infty} \alpha_{pq}=\sum_{q=1}^{\infty} \alpha_{pq}=1$. The minmax principle \cite{RS-80-4} implies  
$|\lambda_q[\varrho_k-\varrho]| \leq \|\varrho_k-\varrho \|_{\calL(L^2)}$, $\forall q \geq 1 $, so that
$$
|\lambda_p^k-\lambda_p| \leq \|\varrho_k-\varrho \|_{\calL(L^2)}\leq \|\varrho_k-\varrho\|_{\calJ_1}, \qquad p \geq 1,
$$ 
which gives the desired convergence property.

Let us now prove the convergence of eigenfunctions by following \cite{kato} and \cite{GB-08}. Let $\sigma(\varrho)$ be the spectrum of $\varrho$ and consider an eigenvalue $\lambda_p$ of $\rho$ with multiplicity $m_p$. Let $d_p$ be the distance between $\lambda_p$ and the closest different eigenvalue, 
$$d_p=\min_{\mu \in \sigma(\varrho), \mu \neq \lambda_p}|\lambda_{p}-\mu|,$$
and denote by $\Gamma$ the circle of radius $\frac{d_p}{2}$ centered at $\lambda_p$. Assume $k$ is large enough so that 
\begin{equation} \label{diffrho}
 \|\varrho_k-\varrho \|_{\calL(L^2)} < \frac{d_p}{2}.
\end{equation} 
Then according to \cite{kato}, theorem IV.3.18, (see also example 3.20), there are exactly $m_p$ (repeated) eigenvalues of $\varrho_k$ included in $\Gamma$ and denote by $\phi_{p,l}^k$, $l=1,\cdots, m_p$ the associated eigenfunctions. If $(\phi_{p,l})_{l=1,\cdots, m_p}$ denote the eigenfunctions associated to $\lambda_p$, we construct an operator $U^k_p$ such that
$$\phi_{p,l}^k=U^k_p \phi_{p,l}, \qquad 1 \leq l \leq m_p, \qquad \textrm{and} \qquad U^k_p \to I \textrm{ in } \calL(L^2) \textrm{ as } k \to \infty.$$
To do so, let $P_p[\varrho]$ be the projection operator onto the spectral components of $\varrho$ inside $\Gamma$,
$$
P_p[\varrho]=\frac{1}{ 2 i \pi}\int_{\Gamma} (zI-\varrho)^{-1} dz.
$$
According to \cite{kato}, II.4.2, Remark 4.4, if 
\begin{equation} \label{diffP}
\| P_p[\varrho]-P_p[\varrho_k]\|_{\calL(L^2)}<1,
\end{equation}
 then an expression of $U^k_p$ can be given by
\begin{equation} \label{expU}
U^k_p=\left(I-(P_p[\varrho_k]-P_p[\varrho])^2\right)^{- \frac{1}{2}} \left( P_p[\varrho_k] P_p[\varrho]+(I-P_p[\varrho_k])(I-P_p[\varrho])\right).
\end{equation}
Let us verify first that $\fref{diffP}$ holds for $k$ large enough. We have
\begin{eqnarray*}
P_p[\varrho]-P_p[\varrho_k]&=&\frac{1}{ 2 i \pi}\int_{\Gamma} ((zI-\varrho)^{-1}-(zI-\varrho_k)^{-1}) dz,\\
&=&\frac{1}{ 2 i \pi}\int_{\Gamma} (zI-\varrho)^{-1}(\varrho-\varrho_k)(zI-\varrho_k)^{-1}dz.
\end{eqnarray*}
{}From the definition of $d_p$, we have
\be
\label{tu1}
\sup_{z \in \Gamma } \| (zI-\varrho)^{-1} \|_{\calL(L^2)}=\frac{2}{d_p}.
\ee
Moreover, owing \fref{diffrho} and noticing that
$$(zI-\varrho_k)^{-1}=(I+(zI-\varrho)^{-1}(\varrho-\varrho_k))^{-1}(zI-\varrho)^{-1}$$ 
and 
\begin{eqnarray*}
\left\| (I+(zI-\varrho)^{-1}(\varrho-\varrho_k))^{-1} \right\|_{\calL(L^2)} &\leq& \left(1-\| (zI-\varrho)^{-1} (\varrho_k-\varrho )\|_{\calL(L^2)}\right)^{-1},\\
& \leq& \left(1- \frac{2}{d_p}\|\varrho_k-\varrho \|_{\calL(L^2)}\right)^{-1},
\end{eqnarray*}
we conclude that
\be
\label{tu2}
\sup_{z \in \Gamma } \| (zI-\varrho_k)^{-1} \|_{\calL(L^2)} \leq \frac{1}{\frac{d_p}{2}- \|\varrho_k-\varrho \|_{\calL(L^2)}}.
\ee
Therefore, \fref{tu1} and \fref{tu2} imply the inequality
$$
\| P_p[\varrho]-P_p[\varrho_k]\|_{\calL(L^2)}\leq \frac{\|\varrho_k-\varrho \|_{\calL(L^2)}}{\frac{d_p}{2}- \|\varrho_k-\varrho \|_{\calL(L^2)}}.
$$
Assuming $k$ is large enough so that $\|\varrho_k-\varrho \|_{\calL(L^2)} < \frac{d_p}{4}$, the above inequality yields the desired result since
\begin{equation} \label{diffP2}
\| P_p[\varrho]-P_p[\varrho_k]\|_{\calL(L^2)} \leq \frac{4}{d_p} \|\varrho_k-\varrho \|_{\calL(L^2)}<1.
\end{equation}
Let us prove now that  $U^k_p \to I$ in $\calL(L^2)$ as $k \to \infty$. First, remarking that $P_p[\varrho]=P_p[\varrho]^2$ and $P_p[\varrho_k]=P_p[\varrho_k]^2$ since both are projections, \fref{expU} can be recast as 
$$
U^k_p=\left(I-(P_p[\varrho_k]-P_p[\varrho])^2\right)^{- \frac{1}{2}}
 \left( I+P_p[\varrho_k] (P_p[\varrho]-P_p[\varrho_k])+(P_p[\varrho_k]-P_p[\varrho])P_p[\varrho]\right).
$$
Let $\delta:=\| P_p[\varrho]-P_p[\varrho_k]\|_{\calL(L^2)}<1$. Then
$$
\left\| \left(I-(P_p[\varrho^k]-P_p[\varrho])^2\right)^{-\frac{1}{2}} \right\|_{\calL(L^2)} \leq (1-\delta^2)^{- \frac{1}{2}},
$$ 
and, together with \fref{diffP2},
\begin{eqnarray*}
\left\|U^k_p -I \right\|_{\calL(L^2)}& \leq& (1-\delta^2)^{- \frac{1}{2}} \left( \left\| 
P_p[\varrho_k] (P_p[\varrho]-P_p[\varrho_k])\right\|_{\calL(L^2)} \right.\\
 &&\left. +\left\|  (P_p[\varrho_k]-P_p[\varrho])P_p[\varrho]\right\|_{\calL(L^2)}+\left\|(P_p[\varrho_k]-P_p[\varrho])^2\right\|_{\calL(L^2)} \right),\\
& \leq& (1-\delta^2)^{- \frac{1}{2}} \left(2\delta +\delta^2 \right),\\
& \leq& C_p   \|\varrho_k-\varrho \|_{\calL(L^2)},
\end{eqnarray*}
where the constant $C_p$ does not depend on $k$ for $k$ large enough. It thus follows that
$$
\|\phi_{p,l}^k-\phi_{p,l}\|_{L^2} \leq C_p  \|\varrho_k-\varrho \|_{\calL(L^2)} \to 0  \textrm{ as } k \to \infty.
$$
This ends the proof of the lemma.
\end{proof}

\bs 
\ni
{\em Proof of Lemma \pref{lemdiff}.}
Let $\Gamma$ be an oriented curve in $(-\eta, +\infty) \times \RR$ that contains the interval $(-\frac{\eta}{2}, 2\| \varrho\|)$. Let $t \in [-t_0,t_0]$, with $2t_0 \| \omega \| < \min(\eta,\|\varrho\|)$. For such values of $t$, the spectrum of $\varrho+t \omega$ is included in the interval $(-\frac{\eta}{2}, 2\| \varrho\|)$. 

Since $\beta_\eta$ is holomorphic in $(-\eta, \infty) \times \RR$, one can define $\beta_\eta(\varrho+t\omega)$ and $\beta_\eta(\varrho)$ by
\begin{eqnarray*}
\beta_\eta(\varrho)&=&\frac{1}{ 2 i \pi}\int_{\Gamma} \beta_\eta(z)(zI-\varrho)^{-1} dz,\\
\beta_\eta(\varrho +t\omega)&=&\frac{1}{ 2 i \pi}\int_{\Gamma} \beta_\eta(z)(zI-\varrho-t\omega)^{-1} dz.
\end{eqnarray*}
Let $|t| \in [0, \min(t_0,t_1)]$, where
 $$ t_1\, \,\textrm{dist}(\Gamma,\sigma(\varrho))^{-1}\, \|\omega\| <1,$$
$\sigma(\varrho)$ denoting the spectrum of $\varrho$.
We have
\begin{eqnarray*}
(zI-\varrho-t\omega)^{-1}-(zI-\varrho)^{-1}&=&(zI-\varrho)^{-1} \left(I-t \omega (zI-\varrho)^{-1}  \right)^{-1}-(zI-\varrho)^{-1}\\
&=&(zI-\varrho)^{-1} \sum_{k \in \NN^*} (t \omega (zI-\varrho)^{-1})^k
\end{eqnarray*}
where the latter serie is normally converging in $\calJ_1$. Indeed, first, 
\begin{eqnarray*}
\|  (t \omega (zI-\varrho)^{-1})^k \|_{\calJ_1} &\leq &  |t|^k \| \omega (zI-\varrho)^{-1}\|^k_{\calJ_1},
\end{eqnarray*}
so that we only need to estimate $ \omega (zI-\varrho)^{-1}$ in $\calJ_1$. Now, since $\varrho$ is self-adjoint,
$$
 \|(zI-\varrho)^{-1}\| = \textrm{dist}(z,\sigma(\varrho))^{-1} \leq  \textrm{dist}(\Gamma,\sigma(\varrho))^{-1},
$$
 and it comes 
\begin{eqnarray*}
|t| \| \omega (zI-\varrho)^{-1}\|_{\calJ_1} &\leq &  |t|  \|(zI-\varrho)^{-1}\|\| \omega\|_{\calJ_1},\\
&\leq &|t|  \, \textrm{dist}(\Gamma,\sigma(\varrho))^{-1} \|\omega\|_{\calJ_1},\\
&< &1, \qquad \forall |t| \leq t_1.
\end{eqnarray*}
We thus can write
$$
t^{-1}\left((zI- \varrho-t\omega)^{-1}-(zI-\varrho)^{-1}\right)=(zI-\varrho)^{-1} \omega (zI-\varrho)^{-1} +t A(t,z),
$$
where the operator $A(t,z)$ is uniformly bounded in $\calJ_1$ with respect to $t$ and $z$ for $|t| \in [0, \min(t_0,t_1)]$ and $z \in \Gamma$. Hence,
\begin{eqnarray*}
t^{-1} \left[\Tr \beta_\eta(\varrho+t\omega) - \Tr \beta_\eta(\varrho)\right]&=&\frac{1}{2 i \pi} \Tr \int_{\Gamma} \beta_\eta(z) (zI-\varrho)^{-1} \omega (zI-\varrho)^{-1} dz\\
&&+\frac{t}{2 i \pi} \Tr \int_{\Gamma} \beta_\eta(z) A(t,z) \, dz.
\end{eqnarray*}
The two expressions of the right-hand side are well-defined and we have
\begin{eqnarray*}
\left| \Tr \int_{\Gamma} \beta_\eta(z) (zI-\varrho)^{-1} \omega (zI-\varrho)^{-1} dz\right|&\leq &\int_{\Gamma} |\beta_\eta(z) | \| (zI-\varrho)^{-1} \omega (zI-\varrho)^{-1} \|_{\calJ_1}dz,\\
&\leq &\textrm{dist}(\Gamma,\sigma(\varrho))^{-2}\, \| \omega \|\,  \int_{\Gamma} |\beta_\eta(z)| dz,\\
\left| \Tr \int_{\Gamma} \beta_\eta(z) A(t,z) \, dz \right| &\leq& \sup_{z \in \Gamma }\|A(t,z)\|_{\calJ_1} \int_\Gamma |\beta_\eta(z) | dz\leq C_1,\\
\end{eqnarray*}
where $C_1$ is independent of $t\in [0,\min(t_0,t_1)]$. Hence
\begin{eqnarray*}
\lim_{t \to 0^+}t^{-1} \left[\Tr \beta_\eta(\varrho+t\omega) - \Tr \beta_\eta(\varrho)\right]&=&\frac{1}{2 i \pi} \Tr \int_{\Gamma} \beta_\eta(z) (zI-\varrho)^{-1} \omega (zI-\varrho)^{-1} dz,\\
&=&\frac{1}{2 i \pi}  \sum_{i \in \NN^*} (\phi_i, \omega\phi_i) \int_{\Gamma} \beta_\eta(z) (z-\rho_i)^{-2}  dz,
\end{eqnarray*}
where $(\rho_i,\phi_i)_{i \in \NN}$ denote  the spectral elements of $\varrho$. Standard complex analysis then implies that
$$
\frac{1}{2 i \pi} \int_{\Gamma} \beta_\eta(z) (z-\rho_i)^{-2}  dz=\beta_\eta'(\rho_i).
$$
We therefore get the following expression of the G\^ateaux derivative:
\begin{equation} \label{eqDF}
D\widetilde F_\eta(\varrho) (\omega) =  \sum_{i \in \NN^*} \beta_\eta'(\rho_i) (\phi_i, \omega\phi_i).   
\end{equation}
The serie is absolutely converging since $\beta'_\eta(s)=\log(s+\eta)$ is locally bounded on $\RR_+$ (recall that $\eta>0$):
$$
|\beta_\eta'(\rho_i) (\phi_i, \omega \phi_i) | \leq C|(\phi_i, \omega \phi_i)|,
$$
and since we have assumed that $\omega\in \calJ_1$. Finally, to identify the derivative, it suffices to notice that
$$
\sum_{i \in \NN^*} \beta_\eta'(\rho_i) (\phi_i, \omega\phi_i)=\sum_{i \in \NN^*}  (\phi_i, \omega\beta_\eta'(\varrho)\phi_i)=\Tr (\omega\beta_\eta'(\varrho))=\Tr(\beta_\eta'(\varrho)\omega).
$$
The proof of Lemma \ref{lemdiff} is complete.
\qed

\end{appendix}
 
 \ms
 \bs
\ni
{\bf Acknowledgements.}
The authors were supported by the Agence Nationale de la Recherche, ANR project QUATRAIN. F. M\'ehats was also supported by the INRIA project IPSO. The authors wish to thank the anonymous referee for his careful reading and helpful remarks.


\begin{thebibliography}{10}

\bibitem{arnold}
{\sc A.~Arnold}, {\em Self-consistent relaxation-time models in quantum
  mechanics}, Comm. Partial Differential Equations, 21 (1996), pp.~473--506.

\bibitem{GB-08}
{\sc G.~Bal}, {\em Central limits and homogenization in random media},
  Multiscale Model. Simul., 7 (2008), pp.~677--702.

\bibitem{barletti-mehats}
{\sc L.~Barletti and F.~M\'ehats}, {\em Quantum drift-diffusion modeling of
  spin transport in nanostructures}, preprint,  (2009).

\bibitem{QSHE}
{\sc J.-P. Bourgade, P.~Degond, F.~M{\'e}hats, and C.~Ringhofer}, {\em On
  quantum extensions to classical spherical harmonics
  expansion/{F}okker-{P}lanck models}, J. Math. Phys., 47 (2006), pp.~043302,
  26.

\bibitem{Brezis}
{\sc H.~Brezis}, {\em Analyse fonctionnelle}, Collection Math\'ematiques
  Appliqu\'ees pour la Ma\^\i trise. [Collection of Applied Mathematics for the
  Master's Degree], Masson, Paris, 1983.
\newblock Th{\'e}orie et applications. [Theory and applications].

\bibitem{brull-mehats}
{\sc S.~Brull and F.~M\'ehats}, {\em Derivation of viscous correction terms for
  the isothermal quantum euler model}, ZAMM,  (to appear).

\bibitem{QDD-JCP}
{\sc P.~Degond, S.~Gallego, and F.~M{\'e}hats}, {\em An entropic quantum
  drift-diffusion model for electron transport in resonant tunneling diodes},
  J. Comput. Phys., 221 (2007), pp.~226--249.

\bibitem{isotherme}
\leavevmode\vrule height 2pt depth -1.6pt width 23pt, {\em Isothermal quantum
  hydrodynamics: derivation, asymptotic analysis, and simulation}, Multiscale
  Model. Simul., 6 (2007), pp.~246--272.

\bibitem{QHD-CMS}
\leavevmode\vrule height 2pt depth -1.6pt width 23pt, {\em On quantum
  hydrodynamic and quantum energy transport models}, Commun. Math. Sci., 5
  (2007), pp.~887--908.

\bibitem{DGMRlivre}
{\sc P.~Degond, S.~Gallego, F.~M{\'e}hats, and C.~Ringhofer}, {\em Quantum
  hydrodynamic and diffusion models derived from the entropy principle}, in
  Quantum transport, vol.~1946 of Lecture Notes in Math., Springer, Berlin,
  2008, pp.~111--168.

\bibitem{QET}
{\sc P.~Degond, F.~M{\'e}hats, and C.~Ringhofer}, {\em Quantum energy-transport
  and drift-diffusion models}, J. Stat. Phys., 118 (2005), pp.~625--667.

\bibitem{QHD-review}
\leavevmode\vrule height 2pt depth -1.6pt width 23pt, {\em Quantum hydrodynamic
  models derived from the entropy principle}, in Nonlinear partial differential
  equations and related analysis, vol.~371 of Contemp. Math., Amer. Math. Soc.,
  Providence, RI, 2005, pp.~107--131.

\bibitem{DR}
{\sc P.~Degond and C.~Ringhofer}, {\em Quantum moment hydrodynamics and the
  entropy principle}, J. Statist. Phys., 112 (2003), pp.~587--628.

\bibitem{Dolbeault-Loss}
{\sc J.~Dolbeault, P.~Felmer, M.~Loss, and E.~Paturel}, {\em Lieb-{T}hirring
  type inequalities and {G}agliardo-{N}irenberg inequalities for systems}, J.
  Funct. Anal., 238 (2006), pp.~193--220.

\bibitem{DFM}
{\sc J.~Dolbeault, P.~Felmer, and J.~Mayorga-Zambrano}, {\em Compactness
  properties for trace-class operators and applications to quantum mechanics},
  Monatshefte f{\"u}r Mathematik, 155 (2008), pp.~43--66.

\bibitem{QDD-SIAM}
{\sc S.~Gallego and F.~M{\'e}hats}, {\em Entropic discretization of a quantum
  drift-diffusion model}, SIAM J. Numer. Anal., 43 (2005), pp.~1828--1849.

\bibitem{golse}
{\sc F.~Golse}, {\em The {B}oltzmann equation and its hydrodynamic limits}, in
  Evolutionary equations. {V}ol. {II}, Handb. Differ. Equ.,
  Elsevier/North-Holland, Amsterdam, 2005, pp.~159--301.

\bibitem{golse-lsr}
{\sc F.~Golse and L.~Saint-Raymond}, {\em Hydrodynamic limits for the
  {B}oltzmann equation}, Riv. Mat. Univ. Parma (7), 4** (2005), pp.~1--144.

\bibitem{hugenholtz}
{\sc N.~M. Hugenholtz}, {\em On the inverse problem in statistical mechanics},
  Comm. Math. Phys., 85 (1982), pp.~27--38.

\bibitem{jungel2}
{\sc A.~J\"ungel}, {\em Global weak solutions to compressible navier-stokes
  equations for quantum fluids}, submitted for publication,  (to appear).

\bibitem{jungel}
\leavevmode\vrule height 2pt depth -1.6pt width 23pt, {\em New velocity
  formulation of euler equations with third-order derivatives: application to
  viscous korteweg-type and quantum navier-stokes models}, submitted for
  publication,  (to appear).

\bibitem{jungel-matthes}
{\sc A.~J{\"u}ngel and D.~Matthes}, {\em A derivation of the isothermal quantum
  hydrodynamic equations using entropy minimization}, ZAMM Z. Angew. Math.
  Mech., 85 (2005), pp.~806--814.

\bibitem{jungel-matthes-milisic}
{\sc A.~J{\"u}ngel, D.~Matthes, and J.~P. Mili{\v{s}}i{\'c}}, {\em Derivation
  of new quantum hydrodynamic equations using entropy minimization}, SIAM J.
  Appl. Math., 67 (2006), pp.~46--68.

\bibitem{junk}
{\sc M.~Junk}, {\em Domain of definition of {L}evermore's five-moment system},
  J. Statist. Phys., 93 (1998), pp.~1143--1167.

\bibitem{kato}
{\sc T.~Kato}, {\em Perturbation theory for linear operators}, Classics in
  Mathematics, Springer-Verlag, Berlin, 1995.
\newblock Reprint of the 1980 edition.

\bibitem{lemm}
{\sc J.~C. Lemm, J.~Uhlig, and A.~Weigunya}, {\em Bayesian approach to inverse
  quantum statistics: Reconstruction of potentials in the feynman path integral
  representation of quantum theory}, Eur. Phys. J. B., 46 (2005), pp.~41--54.

\bibitem{levermore}
{\sc C.~D. Levermore}, {\em Moment closure hierarchies for kinetic theories},
  J. Statist. Phys., 83 (1996), pp.~1021--1065.

\bibitem{Lieb-Loss}
{\sc E.~H. Lieb and M.~Loss}, {\em Analysis}, vol.~14 of Graduate Studies in
  Mathematics, American Mathematical Society, Providence, RI, second~ed., 2001.

\bibitem{ovy}
{\sc S.~Olla, S.~R.~S.~Varadhan, and H.-T.~Yau}, {\em Hydrodynamical limit for a {H}amiltonian system with weak noise},
  Comm. Math. Phys., 155 (1993), pp.~523--560.

\bibitem{RS-80-I}
{\sc M.~Reed and B.~Simon}, {\em Methods of modern mathematical physics. I.
  Functional analysis}, Academic Press, Inc., New York, second~ed., 1980.

\bibitem{RS-80-4}
\leavevmode\vrule height 2pt depth -1.6pt width 23pt, {\em Methods of modern
  mathematical physics. IV. Analysis of operators}, Academic Press, Inc., New
  York, second~ed., 1980.

\bibitem{ringhofer}
{\sc C.~Ringhofer}, {\em Sub-band diffusion models for quantum transport in a
  strong force regime}, preprint,  (2009).

\bibitem{Simon-trace}
{\sc B.~Simon}, {\em Trace ideals and their applications}, vol.~120 of
  Mathematical Surveys and Monographs, American Mathematical Society,
  Providence, RI, second~ed., 2005.

\end{thebibliography}

\end{document}